\documentclass{article}

\usepackage{natbib}
\usepackage{amsmath, amsthm, amssymb}
\usepackage[utf8]{inputenc}
\usepackage[T1]{fontenc}
\usepackage[english]{babel}
\usepackage{lmodern}
\usepackage{mathtools}            
\usepackage[disable]{todonotes}   
\usepackage{delarray,array}       

\usepackage{pdflscape}
\usepackage{lscape}

\usepackage{multirow}
\usepackage{longtable}
\usepackage{booktabs}

\usepackage{caption}
\usepackage{subcaption}

\usepackage{algpseudocode}
\usepackage{algorithm}

\usepackage{geometry}
\theoremstyle{definition}
\newtheorem{definition}{Definition}
\newtheorem{example}{Example}

\theoremstyle{remark}
\newtheorem{observation}{Remark}

\theoremstyle{plain}
\newtheorem{proposition}{Proposition}

\usepackage{authblk}

\definecolor{miamarillo}{RGB}{238,165,54}

\newcounter{todocounter}
\newcommand{\todonum}[2][]{\stepcounter{todocounter}\todo[#1,color=miamarillo]{\textbf{\thetodocounter}: #2}}

\newcommand{\todoin}[1]{\todonum[inline,color=miamarillo]{#1}}

\newcommand{\jhat}{\hat{\jmath}}

\newcommand{\xmip}{x^*_\text{MIP}}
\newcommand{\xmsp}{x^*_\text{MSP}}

\newcommand{\zmip}{z^*_\text{MIP}}
\newcommand{\zmsp}{z^*_\text{MSP}}

\usepackage[hyperfootnotes=false]{hyperref}
\hypersetup{
	colorlinks=true,
	linkcolor=blue,
	citecolor=[RGB]{0,153,51},
	filecolor=magenta,
	urlcolor=cyan,
	unicode=false,
	pdfnewwindow=true,
}

\begin{document}

\title{A risk-aversion approach for the Multiobjective Stochastic Programming problem}

\author[1,*]{Javier León}
\author[2]{Justo Puerto}
\author[1]{Begoña Vitoriano}

\affil[1]{HUM-LOG Research Group, Instituto de Matemática Interdisciplinar (IMI), Facultad de Ciencias Matemáticas, Universidad Complutense de Madrid, Plaza de las Ciencias 3, Madrid, 28040, Spain}
\affil[2]{Mathematical Research Institute (IMUS), University of Seville, Sevilla, Spain}
\affil[ ]{\url{javileon@ucm.es}, \url{puerto@us.es}, \url{bvitoriano@mat.ucm.es}}
\affil[*]{Corresponding author}

\maketitle
\begin{abstract}
	Multiobjective stochastic programming is a field well located to tackle problems arising in emergencies, given that uncertainty and multiple objectives are usually present in such problems. A new concept of solution is proposed in this work, especially designed for risk-aversion solutions. A linear programming model is presented to obtain such solution.
\end{abstract}



\section{Introduction}
\todoin{Vida real, múltiples criterios, aversión al riesgo}

Decision making is never easy, yet we often have to make decisions. Emergencies and disaster management are fields in which many difficulties often arise, such as high uncertainty and multiple conflicting objectives. To overcome such difficulties, risk-aversion decisions are usually sought. Risk-aversion is the attitude for which we prefer to lower uncertainty rather than gambling extreme outcomes (positive or negative).

Risk-aversion, although typically studied in problems with uncertainty, can as well be considered when making decisions with multiple criteria. For instance, in the field of disaster management, solutions that are sufficiently good for all criteria are usually preferred than others that perform exceptionally good for some criteria but inadequately for the others.

Multicriteria decision making (MCDM) is a field worth of consideration when studying real-world problems. Such is the case that MCDM techniques have been recently used for solving problems as varied as:
disaster management \citep{Gutjahr2016351,Ferrer2018501}, engineering \citep{Sun2018602}, finance \citep{Karsu2015343,Angilella2015540}, forest planning \citep{Fotakis20151}, healthcare \citep{Guido2017270}, location of waste facilities \citep{Eiselt2015305}, police districting \citep{Liberatore2016}, route planning \citep{Bast201619}, train scheduling \citep{Sama2015146} or urban planning \citep{Spina201588,Carli2018223}.

This situation, in which multiple conflicting objectives have to be optimized, has led to the definition of different solution concepts and methodologies. Depending on the problem and the type of solution considered, a specific methodology should be applied.

The concept of efficiency reflects the intuition that for a solution to be acceptable, another cannot exist improving that one in every objective. Multiple notions of efficiency are available. The notation that this paper follows is the given in \cite{Ehrgott2005}.
\begin{definition}[Efficiency, \cite{Ehrgott2005}] \label{def:eficiencia}
	Let $f_1(x),\dots,f_K(x)$ be objective functions to be minimized, and let $X$ be the feasible set. A feasible solution $\hat{x} \in X$ is called:

	\begin{itemize}
		\item \emph{Weakly efficient} if there is no $x \in X$ such that $f(x) < f(\hat{x})$ i.e. $f_k(x) < f_k(\hat{x})$ for all $k=1,\dots,K$.
		\item \emph{Efficient} or \emph{Pareto optimal} if there is no $x \in X$ such that $f_k(x) \le f_k(\hat{x})$ for all $k=1,\dots,K$ and $f_i(x) < f_i(\hat{x})$ for some $i \in \{1,\dots,K\}$.
		\item \emph{Strictly efficient} if there is no $x \in X$, $x \neq \hat{x}$ such that $f(x) \le f(\hat{x})$.	
	\end{itemize}
\end{definition}

Furthermore, the set of efficient solutions is called the \emph{efficient set}, and the image under $f$ of this set is the \emph{nondominated} set. It is reasonable to assume that the solution given to any problem must lie in the efficient set. 


Uncertainty is another feature present in the studied problems, in which risk-averse decisions will be preferred. The most common ways for dealing with the uncertainty are \emph{stochastic programming} and \emph{robust optimization}, in which \emph{fuzzy optimization} is also included \citep{Rommelfanger2004}.

The different approaches for treating uncertainty do not respond to the desires of the modeller, instead, they reflect the nature of the uncertainty.
If the uncertainty comes with an underlying known or estimated probability distribution, then stochastic programming is used.
On the other hand, if uncertainty comes from a lack of precision or semantic uncertainty, then robust optimization is used.
Robust optimization does not assume a known (or existing) distribution \citep{BENTAL19991,chen2007,KLAMROTH2017}.
A recent review of robust optimization is written in \cite{GABREL2014471}.
For an introduction to stochastic programming the reader is referred to \cite{Birge2011}.

Stochastic programming is the widest used technique when there are historical data or information to infer a probability distribution. Moreover, usually discrete distributions are used, calling scenarios the different values. The concepts of \emph{value-at-risk} (VaR) and \emph{conditional value-at-risk} (CVaR) are widely used for quantifying risk (see for instance \cite{YAO20131014,Mansini2015,Liu2017,Dixit2019,FERNANDEZ2019215}).
They are typically defined for losses distributions in finance, where the right tail of the distributions are of interest.


\begin{definition}[CVaR, \cite{ROCKAFELLAR20021443}]
	Given $F_X(x)$ distribution function, and $\beta \in [0,1]$, the $\beta$-CVaR is the conditional expectancy over $\{x: F_X(x) \ge \beta \}$.
\end{definition}


Consider now the following problem, in which multiple objectives to be minimized and uncertainty are included simultaneously:
\begin{equation*}
	\min_{x \in X} \left( f_1(x,\omega), \dots, f_K (x,\omega)\right)
\end{equation*}
The above problem is typically called \emph{multiobjective stochastic programming problem} (MSP), especially if $\omega$, the uncertainty source, has a known probability distribution.

In this paper we introduce a new solution concept in multiobjective stochastic programming based on risk-aversion preferences. Such concept is complemented with a mathematical programming model to efficiently compute it, and computational experiments are performed to assess its strengths.

\paragraph{Structure of paper}
The remaining of this paper is organized as follows. Section~\ref{sec:defs} includes the definition of a novel concept of solution for MSP problems and studies its properties. Section~\ref{sec:firstexample} illustrates a basic example of how this solution can be found if the decision space is finite and small. 

Section~\ref{sec:MPmodel} shows how to obtain such a solution with a linear programming model. An application to the multicriteria knapsack problem is developed in Section~\ref{sec:knap}. 
\section{Multiobjective stochastic programming}\label{sec:defs}
\subsection{Literature review}
\cite{Goicoechea1980} develops PROTRADE method, where utility functions are defined to aggregate objectives into a single objective stochastic problem. The resulting problem is solved with an interactive method, where the decision-maker defines an expected solution and a feasibility probability.
\cite{Leclercq1982} reduces the stochasticity by adding some \emph{good} measures to the list of objectives, such as the mean, variance, or probability of being over/below a threshold.
The resulting multiobjective deterministic problem is solved aggregating the objectives, but it could be solved via other techniques.

\cite{Caballero2004} compare the \emph{stochastic approach} with the \emph{multiobjective approach} when using different techniques. The \emph{stochastic approach} transforms the MSP on a single-objective stochastic problem, while the \emph{multiobjective approach} first reduces the stochasticity transforming the MSP on a deterministic multiobjective problem. They highlight that \emph{``the multiobjective approach \emph{forgets} the possible existence of stochastic dependencies between objectives.''}
\cite{Aouni2005} study stochastic goal programming, where the deviation of the objective functions to some goals set beforehand to stochastic values is minimized.

In \cite{BenAbdelaziz2010} a chance-constrained compromise approach is proposed, with an example presented in \cite{Abdelaziz2007}.
In \cite{Munoz2010} the \emph{INTEREST} method is proposed. It is an interactive reference point method. The decision-maker gives reference levels $u_i$ and probabilities $\beta_i$, hoping to achieve a solution $x^*$ such that $\mathbb{P}\left(f_i(x^*) \le u_i \right) \ge \beta_i$. If this is infeasible, the decision-maker should either increase the reference levels or decrease the probabilities of achievement.
\cite{BenAbdelaziz2012} reviews different solutions methods for the MSP problem, categorizing them as stochastic approach or multiobjective approach.

Some fields where MSP models have been developed are: forest management \citep{ALVAREZMIRANDA201879}, multiple response optimization \citep{DIAZGARCIA2014}, energy generation \citep{Teghem1986,BATH2004}, energy exchange \citep{GAZIJAHANI2018}, capacity investment \citep{Claro201085}, disaster management \citep{Manopiniwes2016,Bastian2016}, portfolio optimization \citep{Tuncerakar2012} and cash management \citep{SalasMolina2019}, among others

\subsection{Definitions and dominance relationship}
The concept of CVaR allows to aggregate several scenarios by just looking at what happens in the worst ones. The \emph{ordered weighted averaging} (OWA) operators are defined in \cite{Yager1988}, and independently in the field of locational analysis \cite{carrizosa94,Nickel1999} under the name of \emph{ordered median function}. These concepts will allow us to aggregate different criteria by looking at the least desirable ones, as a risk-aversion measure.

\begin{definition}[OWA, \cite{Yager1988}]
	Given $a_1,\dots,a_n \in \mathbb{R}$, the \emph{ordered weighted averaging} (OWA) operator with weights $\lambda_1,\dots,\lambda_n$ is defined as:
	$$ OWA(a_1,\dots,a_n) = \sum_i \lambda_i a_{(i)}$$
	where $\left(a_{(1)},\dots,a_{(n)}\right)$ is the ordered vector from largest to smallest $\left(a_1,\dots,a_n\right)$.
\end{definition}

\begin{observation}
	For certain weights, the OWA represents a known quantity:
	\begin{itemize}
		\item If $\lambda_i = \frac{1}{n}$, the resulting OWA is the average of $a$.
		\item If $\lambda_1 = 1$, and $\lambda_j = 0$ for $j>1$, the OWA is the maximum of $a$.
		\item If $\lambda_n = 1$, and $\lambda_j = 0$ for $j<n$, the OWA is the minimum of $a$.
	\end{itemize}
\end{observation}

\cite{YAGER201689} later study how to assign weights for an OWA when criteria have different importances.

\begin{definition}[OWA with importances, \cite{YAGER201689}]\label{def:owaimport}
	Given $a_1,\dots,a_n \in \mathbb{R}$ with importances $u_1,\dots,u_n$ such that $\sum_i u_i = 1$ the weights $\lambda_j$ for the OWA can be calculated with $f$, the \emph{weight generating function} in the following manner:
	\begin{enumerate}
		\item Sort vector $a$ such that $a_{(1)} \ge a_{(2)} \ge \dotsc \ge a_{(n)}$.
		\item With $(\cdot)$ as the order induced by $a$, define $T_j = \sum_{k=1}^j u_{(k)}$.
		\item Let $f$ be a function such that $f:[0,1]\to[0,1]$ and $f(0)=0, f(1)=1$. This function is called \emph{weight generating function}.
		\item Obtain the weights as $\lambda_j = f(T_j) - f(T_{j-1})$.
	\end{enumerate}
\end{definition}

\newcommand{\sub}[1]{_{(#1)}}
\begin{example}[of Definition~\ref{def:owaimport}] \label{ex:weightgenfunct}
	Consider the following weight generating function, for a given $r \in (0,1]$:
	$$ f(x) = \begin{cases}
		\frac{x}{r} & \text{if } x<r\\
		1           & \text{if } x\ge r
	\end{cases}$$

	Let $(\cdot)$ be the order such that $a\sub{1}\ge\dots\ge a\sub{n}$, $u\sub{j}$ the weight associated to $a\sub j$, and also let $T_j = \sum_{k=1}^j u_{(k)}$.
	We shall see know how the weights are obtained from $f$. Let $j^*$ be such that $T_{j^*-1} < r \le T_{j^*}$.
	\begin{itemize}
		\item $\lambda_1 = f(T_1) = f(u\sub{1}) = \frac{u\sub{1}}{r}$, assuming $u\sub{1}<r$
		\item $\lambda_2 = f(T_2) - f(T_1) = f(u\sub{1} + u\sub{2}) - f(u\sub{1})= \frac{u\sub{1}+u\sub{2}}{r} - \frac{u\sub{1}}{r} = \frac{u\sub{2}}{r}$, assuming $u\sub{1} + u\sub{2}<r$
		\item \dots
		\item $\lambda_{j^*} = f(T_{j^*}) - f(T_{j^*-1}) = 1 - \left(\frac{u\sub{1}+u\sub{2}+\dots+u\sub{j^*-1}}{r}\right)$, since $T_{j^*}\ge r$
		\item $\lambda_{j^*+1} = f(T_{j^*+1}) - f(T_{j^*}) = 1 - 1 = 0$
		\item \dots
		\item $\lambda_{n} = f(T_{n}) - f(T_{n-1}) = 1 - 1 = 0$
	\end{itemize}

	Consequently the OWA of $a_1,\dots,a_n$ with importances $u_1,\dots,u_n$ is:
	\begin{align*}
	  OWA = & \frac{u\sub{1}}{r}a\sub{1} + \frac{u\sub{2}}{r}a\sub{2} + \dots + \left[1 - \left(\frac{u\sub{1}+u\sub{2}+\dots+u\sub{j^*-1}}{r}\right)\right] a\sub{j^*} \\
	      = & \frac{u\sub{1}a\sub{1} + u\sub{2}a\sub{2} + \dots +\left(r - u\sub{1} - u\sub{2} - \dots\right)a\sub{j^*}}{r}
	\end{align*}
	That is, the OWA is the average of the worst $a_j$, weighted by their importances, with total importance adding up to $r$
\end{example}

The starting point of this paper is the recurrent idea of representing ordered weighted or ordered median operators by means of $k$-sums. $k$-sums (or $k$-centra in the location analysis literature) are sums of the $k$-largest terms of a vector \citep{Puerto2017}. One can trace back, at least to \cite{KALCSICS2002149}, the use of $k$-sums to represent ordered median objectives. More recent references are for instance \cite{BLANCO20131448,Blanco:2014:RSP:2633217.2633282,PONCE2018127} and \cite{FILIPPI2019156}. This last reference introduces a normalized version of $k$-centrum, named $\beta$-average that will be used in our paper.

Through the remaining of the paper consider that $f_k^j(x)$ are functions to be minimized within a feasible set $X$, with $k = 1,\dots,K$ representing $K$ different objectives with importances $w_k$ and $j = 1,\dots,J$ encoding $J$ different scenarios with probabilities $\pi_j$.

\begin{definition}[$\beta$-average, $g^\beta_k(x)$, \cite{FILIPPI2019156}] \label{def:betaaverage}
	
	Given $\beta \in (0,1]$, for each criterion $k$ it can be defined $g^\beta_k(x)$ which measures the average of $f$ on the worst scenarios $\left(f_k^1(x),\dots,f_k^J(x)\right)$, with accumulated probability equal to $\beta$.
\end{definition}

\begin{observation}[\cite{FILIPPI2019156}]
	Given a value $\beta$, if the sum of the probabilities of the worst scenarios is exactly $\beta$, then the $\beta$-average is exactly $\left( 1 - \beta\right)$-CVaR.
\end{observation}

\begin{example}\label{ex:betaav}
	Consider a point $x$, a fixed criterion $k$ and 5 different scenarios with probabilities $\pi_j$ and values of $f_k^j$ given. Table~\ref{tab:primerosbetaprom} shows the $\beta$-averages for different values of $\beta$, in which the scenarios have been ordered from largest value of $f$ to smallest.
	\begin{itemize}
		\item For $\beta = 0.2$, the scenario $j=1$ is the only one needed to obtain the worst scenario with probability $0.2$, and hence $g^\beta_k(x) = \frac{0.2 \cdot 10}{0.2} = 0.2$.
		\item When $\beta$ equals 0.3 it is necessary to include scenario 2, obtaining a $\beta$-average of $\frac{0.2 \cdot 10 + 0.1 \cdot 7}{0.3} = 9$.
		\item Finally if $\beta = 0.5$ scenario 3 needs to be added as well, but only with the probability needed until reaching $0.5$: $g^\beta_k(x) = \frac{0.2 \cdot 10 + 0.1 \cdot 7 + 0.2\cdot 4}{0.5} = 7$.
	\end{itemize}

	\begin{table}[ht!]
		\centering
		\caption{Small example of $\beta$-average for different values of $\beta$}
		\label{tab:primerosbetaprom}
		\begin{tabular}{r|ccccc|ccc}
		              & \multicolumn{5}{c|}{scenario}  & \multicolumn{3}{c}{$\beta$} \\
			           & 1    & 2   & 3   & 4    & 5    & $0.2$               & $0.3$              & $0.5$\\ \hline
			   $\pi_j$ & 0.2  & 0.1 & 0.3 & 0.25 & 0.15 & \multirow{2}{*}{10} & \multirow{2}{*}{9} & \multirow{2}{*}{7}\\
			$f_k^j(x)$ & 10   & 7   & 4   & 3    & 2    &                     &                    &
		\end{tabular}
	\end{table}
\end{example}

When using the $\beta$-average the functions $f_k^j(x)$ were transformed into $g_k^\beta(x)$, a collection of $K$ functions not depending on the scenario. An OWA will be defined now, via its weight generating function, that will reduce the $K$ $\beta$-averages into a scalar function.
\begin{definition}[$r$-OWA, $O_r(x)$] \label{def:r-OWA}
	Given $x_i \in \mathbb{R}$ with importance $w_i$ ($i=1,\dots,K$, $w_i \ge 0, \sum_i w_i = 1$) and $r \in (0,1]$, the function $O_r(x)$ is defined as the OWA with the following weight generating function:
	$$ f(x) = \begin{cases}
		\frac{x}{r} & \text{if } x<r\\
		1           & \text{if } x\ge r
	\end{cases}$$
\end{definition}
\todoin{We should be careful with the multiple meanings of $w$ (importances for criteria and weights for OWA's). Also with $f_k^j(x)$ as the original functions and the weight generating function}
\begin{observation}
	The definition of $O_r(x)$ is made on a similar manner that the one given of the $\beta$-average (Definition~\ref{def:betaaverage}), but it is done on a context with importances rather than probabilities. Example~\ref{ex:rowa} shows the similarities between both approaches.
\end{observation}

\begin{example}\label{ex:rowa}
	Consider a point $x$ and let $g_k(x)$ be the evaluation of $x$ under 5 different criteria with importances $w_j$. Table~\ref{tab:primerosOWAs} shows the $r$-OWAs for different values of $r$, in which the criteria have been ordered from largest values of $g_k(x)$ to smallest. Consider the case $r=0.5$:
	\begin{enumerate}
		\item As $g_k(x)$ are already ordered for largest to smallest, the values of $T_k$ are:
		$$ T_1 = 0.2,T_2 = 0.2+0.1 = 0.3, T_3 = 0.6, T_4 = 0.85, T_5 = 1$$
		\item The values of $T_k$ under $f$:
		$$ f(T_1) = \frac{0.2}{0.5}, f(T_2) = \frac{0.3}{0.5}, f(T_3) = f(T_4) = f(T_5) = 1 $$
		\item The weights of the OWA:
		$$ \lambda_1 = \frac{0.2}{0.5}, \lambda_2 = \frac{0.3-0.2}{0.5} = \frac{0.1}{0.5}, \lambda_3 = 1-\frac{0.3}{0.5} = \frac{0.2}{0.5}, \lambda_4 = \lambda_5 = 0$$
		\item Consequently the $r$-OWA is:
		$$ r\text{-OWA} = \frac{0.2x_{(1)}+0.1x_{(2)}+0.2x_{(3)}}{0.5} = \frac{0.2\cdot 10+0.1\cdot 7+0.2\cdot 4}{0.5} = 7$$
	\end{enumerate}

	\begin{table}[ht!]
		\centering
		\caption{Small example of $r$-OWA for different values of $r$}
		\label{tab:primerosOWAs}
		\begin{tabular}{r|ccccc|ccc}
		            & \multicolumn{5}{c|}{criterion}  & \multicolumn{3}{c}{$r$} \\
			         & 1    & 2   & 3   & 4    & 5    & $0.2$ & $0.3$ & $0.5$\\ \hline
			$w_k$    & 0.2  & 0.1 & 0.3 & 0.25 & 0.15 & \multirow{2}{*}{10} & \multirow{2}{*}{9} & \multirow{2}{*}{7}\\
			$g_k(x)$ & 10   & 7   & 4   & 3    & 2    & & &
		\end{tabular}
	\end{table}
\end{example}

\begin{observation} \label{obs:maximumlambdas}
	Given $x_1,\dots,x_K$ and its associated importances $w_1,\dots,w_K$, then the $\lambda_k$ of the $r$-OWA are determined in such a way that:

	$$O_r(x) = \max \left\lbrace \frac{\tilde{\lambda}_1 x_1 + \dots +\tilde{\lambda}_K x_K}{r} \mid \tilde{\lambda}_k \le w_k, \sum \tilde{\lambda}_k = r \right\rbrace
	\quad \text{with } \lambda_k = \frac{\tilde{\lambda}_k}{r}
	$$	
\end{observation}
\todoin{Maybe this needs to be explicitly proven}

Given $r, \beta \in (0,1]$ and $x \in X$, let us introduce the function $h^\beta_r(x)$ as the $r$-OWA of the $\beta$-averages. That is:
$$ h^\beta_r(x) = O_r\left(g_1^\beta(x),\dots,g_K^\beta(x)\right)$$

\todoin{I think $\beta$ and $r$ could be given some names. For instance, \emph{uncertainty} and \emph{multicriteria} risk-seeking parameters. When $\beta$ is small very few scenarios are considered, so we are very protected against uncertainty risk. Same goes with $r$.}


\begin{observation}
	If the importance of all criteria is the same ($w_k = \frac{1}{K}$ for all $k$) and $r = \frac{n}{K}$ with $n \in \{1,\dots,K\}$, then the $h^\beta_r(x)$ is the average of the $n$ worst $\beta$-averages. Recall that this is called $n$-centra \citep{NickelPuerto05}.
\end{observation}

\begin{definition}[Dominance] \label{def:dominance}
	Let $x$ and $y$ feasible solutions ($x,y\in X$) and $r,\beta \in (0,1]$. Then $x$ dominates $y$ ($x \succsim y$) if $h^\beta_r(x) \le h^\beta_r(y)$, where $h^\beta_r(x)$ is the $r$-OWA of the $\beta$-averages.
\end{definition}

Definition~\ref{def:dominance} induces a domination relationship with the following properties:
\begin{description}
	\item[Reflexivity] Given $x$, $h^\beta_r(x) \ge h^\beta_r(x)$, and then $x \succsim x$, so $\succsim$ is reflexive.
	\item[Transitiveness] Given $x \succsim y$, $y \succsim z$, we have $h^\beta_r(x)\ge h^\beta_r(y)$ y $h^\beta_r(y)\ge h^\beta_r(z)$, and then $h^\beta_r(x) \ge h^\beta_r(z)$, which leads to $x \succsim z$, and we conclude that $\succsim$ is transitive.
	\item[Antisymmetry] Given $x \succsim y$, $y \succsim x$, we have $h^\beta_r(x)\ge h^\beta_r(y)$ and $h^\beta_r(y)\ge h^\beta_r(x)$, but from $h^\beta_r(x) = h^\beta_r(y)$ it cannot be guaranteed that $x=y$, and hence $\succsim$ is not antisymmetric.
\end{description}

\subsection{Idea of solution and dominance properties}
Consider the multiobjective stochastic programming problem:
\begin{equation*}
	\min_{x \in X} \left( f_1(x,\omega), \dots, f_K (x,\omega)\right) 
\end{equation*}

The previously defined concepts of $\beta$-average and $r$-OWA transform the $MSP$ problem into a deterministic multiple objective problem, and then into a deterministic single objective problem.

\vspace{1em}
\begin{minipage}{0.45\textwidth}
	$$ MSP \rightarrow MOP \rightarrow LP (MIP)$$
	$$ f_k^j(x) \xrightarrow{\beta\text{-average}} g_k^\beta(x) \xrightarrow{r\text{-OWA}} h^\beta_r(x)$$
\end{minipage}%
\begin{minipage}{0.50\textwidth}
	\begin{enumerate}
		\item For every $x \in X$ there is a function $f_k^j$ to be minimized which depends on the scenario $j$ and the criterion $k$.
		\item The problem is transformed into a deterministic one with multiple objectives (MOP) using the $\beta$-average concept.
		\item Computing the $r$-OWA, each $x \in X$ is assigned a scalar. The problem consists of finding the $x$ which minimizes this $h^\beta_r(x)$.
	\end{enumerate}
\end{minipage}
\vspace{1em}

The solution procedure lies into what is usually called a \emph{scalarization approach}. When obtaining a minimizer of $h^\beta_r(x)$ it is also desired that the optimal solution is efficient for the associated MOP problem:

\begin{equation}
	\min_{x \in X} \left(g^\beta_1(x),\dots,g^\beta_K(x)\right) \label{eq:MPproblem} \tag{MOP}
\end{equation}

\begin{proposition}
	Given $x^*$ minimum of $h^\beta_r(x)$ the following statements hold:
	\begin{enumerate}
		\item $x^*$ is not necessarily efficient of the MOP problem.
		\item $x^*$ is weakly efficient of the MOP problem.
		\item If $x^*$ is the only minimum of $h^\beta_r(x)$, then $x^*$ is efficient.
		\item Given $x^*$ not efficient, an alternative $y^*$ can be found on a second phase such that $y^*$ is efficient and $h^\beta_r(x^*) = h^\beta_r(y^*)$.
	\end{enumerate}
\end{proposition}

These properties are known when using scalarization techniques \citep{Ehrgott2005}. Hence only an example of the first statement will be shown.

\begin{example}[$x^*$ is not necessarily efficient]
	Consider the example displayed on Table~\ref{tab:dominatedexample}, in which there are only two feasible solutions, two equiprobable scenarios ($\pi_1 = \pi_2 = \tfrac12$), three equally important criteria ($w_1 = w_2 = w_3 = \tfrac13$), and consider the values of $\beta = \tfrac12$ and $r = \tfrac23$ are taken.

	\begin{table}[ht!]
		\caption{Values of two alternatives for each scenario $j$ and criterion $k$, together with their $\beta$-averages ($\beta = \tfrac12$) and $r$-OWAs ($r = \tfrac23$)}
		\label{tab:dominatedexample}
		\begin{minipage}{.5\linewidth}
			\begin{center}
				\caption{Alternative 1}
					\begin{tabular}{r|ccc}
						& $k_1$ & $k_2$ & $k_3$ \\ \hline
						$j_1$ &  0.80  &  0.40  &  0.30  \\
						$j_2$ &  0.60  &  0.20  &  0.65  \\ \hline
						$\beta$-average & 0.80 & 0.40 & 0.65 \\ \hline
						$r$-OWA & \multicolumn{3}{c}{0.725} \\
					\end{tabular}
			\end{center}
		\end{minipage}%
		\begin{minipage}{.5\linewidth}
			\begin{center}
				\caption{Alternative 2}
					\begin{tabular}{r|ccc}
						& $k_1$ & $k_2$ & $k_3$ \\ \hline
						$j_1$ &  0.70  &  0.45  &  0.65  \\
						$j_2$ &  0.80  &  0.30  &  0.50  \\ \hline
						$\beta$-average & 0.80 & 0.45 & 0.65 \\ \hline
						$r$-OWA & \multicolumn{3}{c}{0.725} \\
					\end{tabular}
			\end{center}
		\end{minipage}
	\end{table}

	The $\beta$-averages are $(0.8,0.4,0.65)$ for the first alternative and $(0.8,0.45,0.65)$ for the second alternative. When computing the function $h^\beta_r$, both alternatives have an objective value of $0.725$. Consequently, even though the second alternative is an optimal solution of $h^\beta_r$, it is not an efficient solution of the MOP problem as its $\beta$-averages are dominated by those of the first alternative.
\end{example}

\subsection{An illustrative example} \label{sec:firstexample}
The solution concept proposed will be now applied, first with a discrete (and small) case. When the solution space is discrete, and all feasible solutions can be explicitly enumerated, the steps are as follows:

\begin{description}
	\item[Step 0] Normalize all objective functions $f^j_k(x)$.
	\item[Step 1] Set values for $\beta,r \in (0,1]$.
	\item[Step 2] For every $x\in X$ and every criterion define $g_k^\beta(x)$ as:
	$$g_k^\beta(x) = \parbox{8cm}{\emph{average of worst scenarios for criterion $k$\\ with probabilities adding up to $\beta$}}$$
	\item[Step 3] Define $h_r^\beta(x)$ as:
	$$h_r^\beta(x) = \parbox{8cm}{\emph{average of worst $g_k^\beta(x)$ values\\with importances adding up to $r$}}$$
	\item[Step 4] Search for $x \in X$ minimizing $h_r^\beta(x)$
\end{description}

Assume a decision space with only four alternatives, evaluated under five different scenarios with six criteria. For each of those alternatives it can be computed the value of the functions $f_k^j(x)$ to be minimized. Table~\ref{tab:ejemplo1:tabla1} shows the values of $f$, evaluated on feasible point $x_1$, by each of the scenarios and criteria considered.

\begin{table}[ht!]
	\centering
	\caption{Values of alternative 1 by scenario ($j$) and criteria ($k$)}
	\label{tab:ejemplo1:tabla1}
	\resizebox{\textwidth}{!}{
	\begin{tabular}{rrr|cccccc|}
		& & & \multicolumn{6}{c|}{criteria} \\
		& & & $w_1 = 0.20$ & $w_2 = 0.10$ & $w_3 = 0.20$ & $w_4 = 0.25$ & $w_5 = 0.15$ & $w_6 = 0.10$ \\
		& & & $k_1$ & $k_2$ & $k_3$ & $k_4$ & $k_5$ & $k_6$ \\ \hline
		\parbox[t]{4mm}{\multirow{5}{*}{\rotatebox[origin=c]{90}{scenarios}}}
		& $\pi_1 = 0.15$ & $j_1$ &  0.51  &  0.27  &  0.39  &  0.45  &  0.75  &  0.76  \\
		& $\pi_2 = 0.20$ & $j_2$ &  0.58  &  0.65  &  0.47  &  0.26  &  0.90  &  0.24  \\
		& $\pi_3 = 0.30$ & $j_3$ &  0.48  &  0.44  &  0.90  &  0.50  &  0.93  &  0.65  \\
		& $\pi_4 = 0.25$ & $j_4$ &  0.76  &  0.18  &  0.01  &  0.90  &  0.56  &  0.02  \\
		& $\pi_5 = 0.10$ & $j_5$ &  0.86  &  0.36  &  0.21  &  0.28  &  0.63  &  0.72  \\ \hline 
	\end{tabular}}
\end{table}

The first step consists on calculating the $\beta$-averages. Let assume a value of $\beta = 0.3$:
\begin{enumerate}
	\item For the first criterion the worst scenario is $j_5$, which has probability $0.1$. The second worst is $j_4$, with a probability of $0.25$. As the sum of those probabilities exceeds the $\beta$ fixed, for computing the $\beta$-average just a probability of $0.2$ is considered:
	$$g_1^\beta(x_1) = \frac{0.1 \cdot 0.86 + 0.2 \cdot 0.76}{0.3} = 0.793$$
	\item $g_2^\beta(x_1) = \left(0.2 \cdot 0.65 + 0.1 \cdot 0.44  \right)/0.3 = 0.580$
	\item $g_3^\beta(x_1) = \left(0.3 \cdot 0.90 \right)/0.3 = 0.900$
	\item $g_4^\beta(x_1) = 0.833$, $g_5^\beta(x_1) = 0.930$, $g_6^\beta(x_1) = 0.728$
\end{enumerate}

The last step is calculating the function $h_r^\beta(x)$, that is, the $r$-OWA of the $\beta$-averages. Table~\ref{tab:inst1alt1} calculates the $r$-OWA, and shows as well the information of the previously calculated $\beta$-averages, when the value of $r = 0.17$ is taken. 

\begin{table}[ht!]
	\centering
	\caption{Values of alternative 1 by scenario ($j$) and criteria ($k$)}
	\label{tab:inst1alt1}
	\resizebox{\textwidth}{!}{
	\begin{tabular}{rrr|cccccc|}
		& & & \multicolumn{6}{c|}{criteria} \\
		& & & $w_1 = 0.20$ & $w_2 = 0.10$ & $w_3 = 0.20$ & $w_4 = 0.25$ & $w_5 = 0.15$ & $w_6 = 0.10$ \\
		& & & $k_1$ & $k_2$ & $k_3$ & $k_4$ & $k_5$ & $k_6$ \\ \hline
		\parbox[t]{4mm}{\multirow{5}{*}{\rotatebox[origin=c]{90}{scenarios}}}
		& $\pi_1 = 0.15$ & $j_1$ &  0.51  &  0.27  &  0.39  &  0.45  &  0.75  &  0.76  \\
		& $\pi_2 = 0.20$ & $j_2$ &  0.58  &  0.65  &  0.47  &  0.26  &  0.90  &  0.24  \\
		& $\pi_3 = 0.30$ & $j_3$ &  0.48  &  0.44  &  0.90  &  0.50  &  0.93  &  0.65  \\
		& $\pi_4 = 0.25$ & $j_4$ &  0.76  &  0.18  &  0.01  &  0.90  &  0.56  &  0.02  \\
		& $\pi_5 = 0.10$ & $j_5$ &  0.86  &  0.36  &  0.21  &  0.28  &  0.63  &  0.72  \\ \hline 
		\multicolumn{3}{c|}{$\beta$-average, $\beta = 0.30$} 
		            & 0.793 & 0.580 & 0.900 & 0.833 & 0.930 & 0.728 \\ \hline
		\multicolumn{3}{c|}{$r$-OWA, $r = 0.17$} & \multicolumn{6}{c|}{0.927} \\ \hline
	\end{tabular}}
\end{table}

\paragraph{Results} The values of the functions for the other alternatives, as well as its $\beta$-averages and $r$-OWAs are shown in Tables~\ref{tab:inst1alt2}, \ref{tab:inst1alt3} and \ref{tab:inst1alt4}, starting on Page~\pageref{tab:inst1alt2}. A summary of the results can be seen in Table~\ref{tab:inst1allvalues}, where all the $\beta$-averages and $r$-OWAs are shown, determining that the optimal alternative for the values of $\beta$ and $r$ given is Alternative 1.

\begin{table}[ht!]
	\caption{$\beta$-averages and $r$-OWAs for each of the 4 feasible alternatives of the example}
	\label{tab:inst1allvalues}
	\centering
	\begin{tabular}{r|cccccc|c}
		& \multicolumn{6}{c|}{$\beta$-averages} & $r$-OWA \\
		& $g_1^\beta(x)$ & $g_2^\beta(x)$ & $g_3^\beta(x)$ & $g_4^\beta(x)$ & $g_5^\beta(x)$ & $g_6^\beta(x)$ & $h_r^\beta(x)$ \\ \hline
		Alternative 1 & 0.793 & 0.580 & 0.900 & 0.833 & 0.930 & 0.728 & 0.927  \\
		Alternative 2 & 0.930 & 0.832 & 0.703 & 0.820 & 0.660 & 0.770 & 0.930  \\
		Alternative 3 & 0.765 & 0.775 & 0.468 & 0.643 & 0.950 & 0.883 & 0.943  \\
		Alternative 4 & 0.993 & 0.760 & 0.473 & 0.773 & 0.820 & 0.990 & 0.993  \\
	\end{tabular}
\end{table}

Variations on $\beta$ and $r$ yield very different results. Figure~\ref{fig:clusters} shows which of the four alternatives has the lowest $h$ value, depending on the values of $\beta$ and $r$. 

Figure~\ref{fig:hvalues} shows the optimal objective value when varying the parameters $\beta$ and $r$. It can be appreciated how $h$ decreases when $\beta$ and $r$ increase. This is due to the fact that the original $f_k^j$ functions are to be minimized, and the larger the parameters $\beta$ and $r$ are, more favourable scenarios/criteria will take part on the computation of $h_r^\beta(x)$, hence decreasing its optimal value.

\begin{figure}[ht!]
	\centering
	\begin{subfigure}{.5\textwidth}
		\centering
		\includegraphics[width=\linewidth]{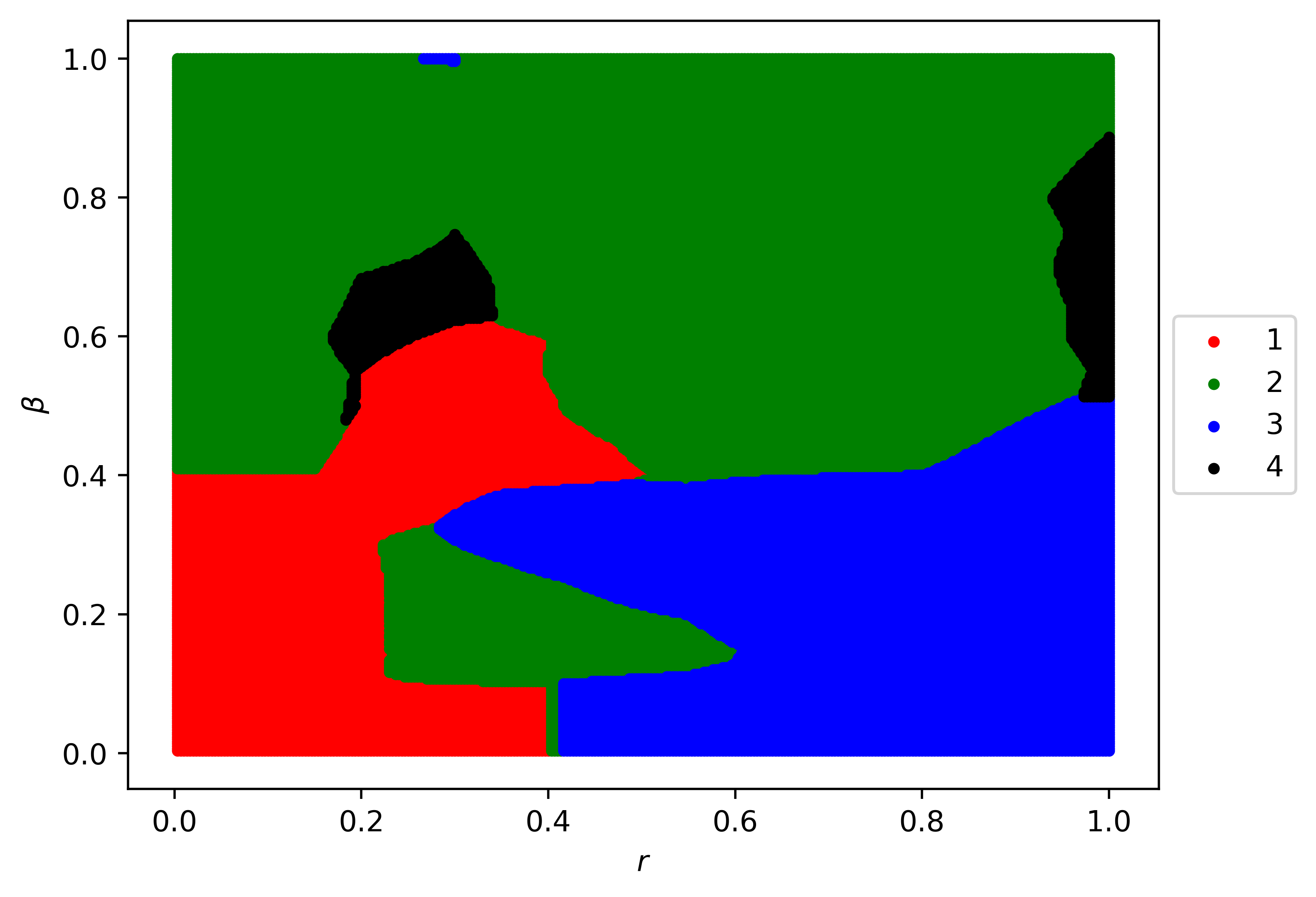}
		\caption{Optimal alternative for some values of $r$ and $\beta$}
		\label{fig:clusters}
	\end{subfigure}%
	\begin{subfigure}{.5\textwidth}
		\centering
		\includegraphics[width=0.9\linewidth]{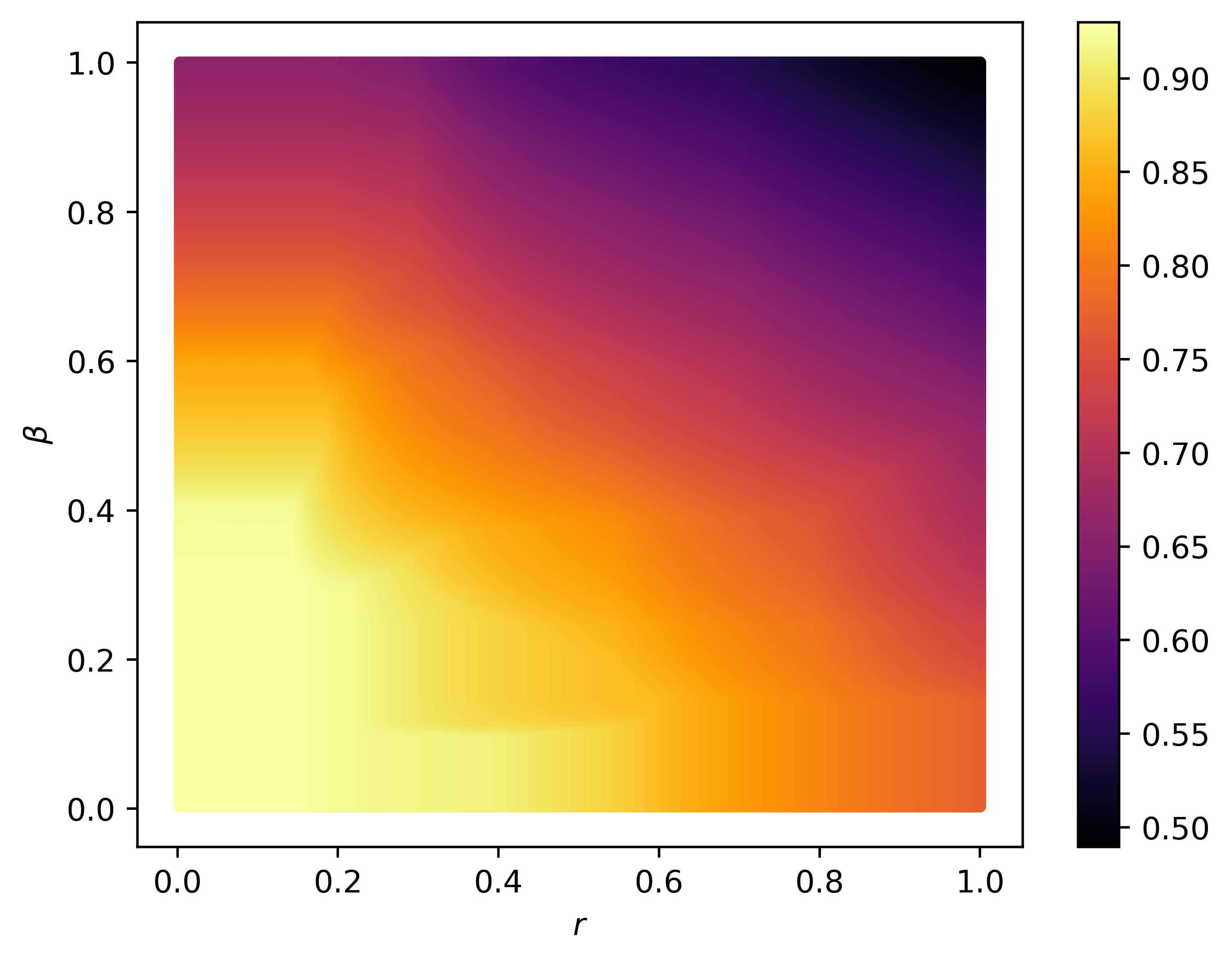}
		\caption{Optimal values of function $h_r^\beta(x)$ for some values of $r$ and $\beta$}
		\label{fig:hvalues}
	\end{subfigure}
	\caption{Results from illustrative example}
	\label{fig:illexample}
\end{figure}
\section{Computing the minimum: continuous case}\label{sec:MPmodel}
A concept of solution was proposed with Definition~\ref{def:dominance}. When the functions $f_k^j(x)$ to be minimized are given, a new function $h_r^\beta(x)$ to be minimized is defined, with parameters $\beta$ and $r$ such that $h_r^\beta(x)$ is the $r$-OWA of the $\beta$-averages. If the decision space is sufficiently small, the procedure shown in the above example obtains such a solution.

In this section, a mathematical programming model will be developed to obtain the minimum of $h_r^\beta(x)$ which allows one to obtain the proposed solution for bigger decision spaces, including continuous ones.


\subsection{Mathematical programming model}
Given $k$ and $x \in X$ we have the vector $\left(f_k^1(x),\dots,f_k^J(x)\right)$. Let $\left(f_k^{(1)}(x),\dots,f_k^{(J)}(x)\right)$ be the ordered vector such that $f_k^{(j_1)}(x) \ge f_k^{(j_2)}(x)$ when $j_1 \le j_2$.

Given $\beta \in (0,1]$, let $\jhat$ be the ordered scenario such that:
$$ \sum_{j = 1}^{\jhat} \pi_{(j)} \ge \beta, \qquad \sum_{j = 1}^{\jhat -1} \pi_{(j)} < \beta$$

Alternatively:

$$ f_k^{(1)}(x) \ge f_k^{(2)}(x) \ge \dots \ge f_k^{(\jhat)}(x) \ge f_k^{(\jhat+1)}(x) \ge \dots \ge f_k^{(J)}(x)$$
$$ 1 = \underbrace{\underbrace{\pi_{(1)} + \pi_{(2)} + \dots + \pi_{(\jhat-1)}}_{< \beta} + \pi_{(\jhat)}}_{\ge \beta} + \dots + \pi_{(J)}$$

Also let:
$$ \hat{\pi}_j = \begin{cases}
	\pi_j                                      & j \in \{(1),\dots,(\jhat-1) \} \\
	\beta - \sum_{j=(1)}^{j=(\jhat -1)} \pi_j & j = \jhat \\
	0                                               & \text{otherwise}
\end{cases}$$
The definition of $\hat{\pi}_{\jhat}$ is made in such a way that $\sum_j \hat{\pi}_j = \beta$. In this way, the average of the $\beta$ worst values can be computed as $\frac{1}{\beta}\sum_{j=1}^J \hat{\pi}_j f_k^{(j)} (x)$, which coincides with the definition of $\beta$-average (Definition~\ref{definicion_gbeta}). This computation can be written as the following optimization problem:
\begin{align*}
	\max_{\tilde{u}_j}\quad   & \frac1\beta\sum_{j = 1}^J \tilde{u}_j \cdot f_k^j (x) \\
	\text{s.t.} \quad & \sum_{j=1}^J \tilde{u}_j = \beta           \\
                     & 0 \le \tilde{u}_j \le \pi_j \qquad j = 1,\dots,J 
\end{align*}

A more natural approach would be to consider $u_j = \frac{\tilde{u}_j}{\beta}$. These $u_j$ represent the proportion in which scenario $j$ plays a part on the aggregated $\beta$-average. Introducing that change, the model is:
\begin{align*}
	\max_{u_j}\quad   & \sum_{j = 1}^J u_j \cdot f_k^j (x) \\
	\text{s.t.} \quad & \sum_{j=1}^J u_j = 1           \\
                     & 0 \le u_j \le \frac{\pi_j}{\beta} \qquad j = 1,\dots,J 
\end{align*}

The dual formulation is:
\begin{align}\label{eq:defg}
	\begin{split}
		\min_{z,y_j}\quad   & z + \sum_{j=1}^J \frac{\pi_j}{\beta} y_j \\
		\text{s.t.} \quad   & z + y_j \ge f_k^j(x) \qquad j = 1,\dots,J           \\
   	                    & z \text{ free}, y_j \ge 0 
	\end{split}
\end{align}

And hence finding the $x \in X$ which minimizes the average of the worst $\beta$ scenarios for a given $k$ is:
\begin{equation*}
	\begin{array}[t]{c}
		\min\limits_{x \in X}
	\end{array}
	\begin{array}[t]({rl})
		\max\limits_{\tilde{u}_j}  & \frac1\beta\sum_{j = 1}^J \tilde{u}_j f_k^j (x) \\
		\text{s.t.} & \sum_{j=1}^J \tilde{u}_j = \beta  \\
		& 0 \le \tilde{u}_j \le \pi_j \qquad j=1,\dots,J \\
	\end{array}
\end{equation*}

Or alternatively:
\begin{equation} 
\begin{split}\label{mod:minmin}
	\begin{array}[t]{c}
		\min\limits_{x \in X}
	\end{array}
	\begin{array}[t]({rl})
		\min\limits_{z,y_j}  & z + \sum_{j=1}^J \frac{\pi_j}{\beta} y_j \\
		\text{s.t.} & z + y_j \ge f_k^j(x) \qquad j = 1,\dots,J  \\
		& z \text{ free}, y_j \ge 0 \qquad j = 1,\dots,J\\
	\end{array}
\end{split}
\end{equation}

Which is equivalent to:
\begin{subequations}\label{modelo_kfijo}
\begin{alignat}{3}
	\min_{z,y_j,x}\quad   & z + \sum_{j=1}^J \frac{\pi_j}{\beta} y_j & \label{kfijofobj}\\
	\text{s.t.} \quad   & z + y_j \ge f_k^j(x) \quad & j = 1,\dots,J \label{kfijor2}\\
                       & z \text{ free}, y_j \ge 0 \quad & j = 1,\dots,J \nonumber\\
	                    & x \in X \nonumber &
\end{alignat}
\end{subequations}
\begin{observation}
	Models \eqref{mod:minmin} and \eqref{modelo_kfijo} are equivalent, as for any $x \in X$ chosen in \eqref{modelo_kfijo} the values $z$ and $y_j$ will get as small as allowed by constraint~\eqref{kfijor2}, as this improves the objective function~\eqref{kfijofobj}. Consequently for every $x$, its $\beta$-average will be computed appropriately, and thus \eqref{modelo_kfijo} obtains the $x\in X$ with smallest $\beta$-average, as desired on~\eqref{mod:minmin}. 
\end{observation}

For every $k \in \{1,\dots,K\}$ thanks to the problem~\eqref{eq:defg} the function $g^\beta_k(x)$ can be defined, which measures for each $x\in X$ the $\beta$-average for that criterion, being:
\begin{equation}\label{definicion_gbeta}
\begin{alignedat}{2}
	g^\beta_k(x) \equiv \min_{z_k,y_{kj}}   & z_k + \sum_{j=1}^J \frac{\pi_j}{\beta} y_{kj} & & \\
	\text{s.t.} \; & z_k + y_{kj} \ge f_k^j(x) \quad & & j = 1,\dots,J     \\
               & z_k \text{ free}, y_{kj} \ge 0 \quad & & j = 1,\dots,J 
\end{alignedat}
\end{equation}

The already known approach for single criterion problems ends here. Given that, the next step is finding a ``good'' solution for all $k$. That is:
\[\min_{x \in X} \left(g^\beta_1(x),\dots,g^\beta_K(x)\right)\]

Given $r \in (0,1]$ the $r$-OWA of the $\beta$-averages will be now computed (in accordance with the definition given in Section~\ref{sec:defs}). That is, the solution of the following problem is sought:
\begin{align*}
	\max_{\tilde{t}_k} \quad & \frac1r\sum_k\tilde{t}_k \cdot g^\beta_k(x) \\
	                      & \sum_k \tilde{t}_k = r         \\
	                      & 0 \le \tilde{t}_k \le w_k \qquad k = 1,\dots,K
\end{align*}

Or equivalently: 
\begin{align*}
	\max_{t_k} \quad & \sum_kt_k \cdot g^\beta_k(x) \\
	                      & \sum_k t_k = 1         \\
	                      & 0 \le t_k \le \frac{w_k}{r} \qquad k = 1,\dots,K
\end{align*}

Its dual formulation is:
\begin{alignat*}{3}
	\min_{z,v_k} \quad & z + \sum_k \frac{w_k}{r}v_k & \\
	\text{s.t.}  \quad & z + v_k \ge g^\beta_k(x) \quad & k = 1,\dots,K\\
	                   & z \text{ free}, v_k \ge 0 \quad & k = 1,\dots,K
\end{alignat*}

Replacing the value of $g_k^\beta(x)$ given in~\eqref{definicion_gbeta} the next model is obtained:

\begin{subequations}\label{mod:mininside}
\begin{align}
	\min_{z,v_k} \quad 
	                   & z + \sum_k \frac{w_k}{r}v_k \label{mod:mininside:fobj}\\
	 \text{s.t.} \quad & z + v_k \ge \begin{pmatrix*}[l]
						       \min\limits_{z_k,y_{kj}} & z_k + \sum_{j=1}^J \frac{\pi_j}{\beta} y_{kj} \\
						   	 \quad\text{s.t.} &  z_k + y_{kj} \ge f_k^j(x) \quad\forall j \\
						       & z_k \text{ free}, y_{kj} \ge 0 
	                    \end{pmatrix*} \forall k \label{mod:mininside:min} \\
	                    & z \text{ free}, v_k \ge 0 \qquad \forall k \label{mod:mininside:dom}
\end{align}
\end{subequations}

Model \eqref{mod:mininside} calculates for a given $x\in X$ the $r$-OWA of its $\beta$-averages, which coincides with the notion of the function $h(x)$ given in Section~\ref{sec:defs}. This problem is not explicit in that it contains nested optimization problems in the constraints. For that reason, we propose a single level alternative for $x \in X$ fixed.

Consider the following linear programming model:
\begin{subequations}\label{mod:final}
\begin{alignat}{3}
	\min_{z,v_k,z_k,y_{kj}} \quad & z + \sum_k \frac{w_k}{r}v_k \label{mod:final:fobj} & \\
	\text{s.t.} \quad  & z + v_k \ge z_k + \sum_{j=1}^J \frac{\pi_j}{\beta} y_{kj} \quad & \forall k \label{mod:final:dualcriterios} \\
	   	             & z_k + y_{kj} \ge f_k^j(x) \quad & \forall k,j  \label{mod:final:dualescenarios}  \\
	                   & y_{kj} \ge 0 \quad & \forall k,j   \label{mod:final:dom1} \\
	                   & z_k \text{ free}, v_k \ge 0 \quad & \forall k  \label{mod:final:dom2} \\
	                   & z \text{ free}	 \label{mod:final:dom3} &
\end{alignat}
\end{subequations}



\begin{proposition}\label{prop:transformations}
	Transformation from model~\eqref{mod:mininside} to model~\eqref{mod:final} is valid, in that their optimal solution and objective values coincide. 
\end{proposition}

\begin{proof}
	Let $\left(z^*,v_k^*,z_k^*,y_{kj}^*\right)$ be the optimal solution of model~\eqref{mod:final}. $\left(z^*,v_k^*\right)$ is feasible of model~\eqref{mod:mininside}, and it will be shown that it is also optimal for such model.
	Assume it exists $\left(z',v_k'\right)$ feasible of model~\eqref{mod:mininside} with:
	$$ z' + \sum_k \frac{w_k}{r}v_k' < z^* + \sum_k \frac{w_k}{r}v_k^*$$
	This and constraint~\eqref{mod:final:dualcriterios} implies there exists $k_0$ such that:
	$$ z' + v_{k_0}' < z_{k_0}^* + \sum_{j=1}^J \frac{\pi_j}{\beta} y_{{k_0}j}^*$$
	otherwise $\left(z',v_k',z_k^*,y_{kj}^*\right)$ would be optimal of model~\eqref{mod:final}.
	Since $z_{k_0}^*$ and $y_{{k_0}j}^*$ are feasible of model~\eqref{mod:final} they are also feasible of the model on the RHS of constraint~\eqref{mod:mininside:min}, and thus $z'$ and $v_{k_0}'$ violate constraint~\eqref{mod:mininside:min}.
\end{proof}

Proposition~\ref{prop:transformations} showed that the optimal solutions of models~\eqref{mod:mininside} and \eqref{mod:final} coincide. Proposition~\ref{prop:equivalencia} goes further showing the connection between their feasible sets.

\begin{proposition}\label{prop:equivalencia}
	The feasible set of model~\eqref{mod:mininside} is a projection of the feasible set of model~\eqref{mod:final}.	
\end{proposition}

\begin{proof}
	\mbox{}
	\begin{enumerate}
		\item For each feasible solution $(z,v_k)$ of model~\eqref{mod:mininside} there is at least one feasible solution of model~\eqref{mod:final} with same values $(z,v_k)$, being so the same objective function.

		Let $(z^1,v^1_k)$ a feasible solution of model~\eqref{mod:mininside}, and $(z^*_k,y^*_{kj})$ the optimal solution where the minimum of the right-hand side of equation \eqref{mod:mininside:min} is reached for each $k$. Since constraints \eqref{mod:final:dualcriterios}, \eqref{mod:final:dualescenarios}, \eqref{mod:final:dom1} and \eqref{mod:final:dom2} are satisfied in model~\eqref{mod:mininside}, $(z^1,v^1_k,z^*_k,y^*_{kj})$ is a feasible solution or model~\eqref{mod:final}.

		\item For each feasible solution $(z,v_k,z_k,y_{kj})$ of model~\eqref{mod:final}, $(z,v_k)$ is a feasible solution of model~\eqref{mod:mininside}, being so the same objective function.
		Let $(z^2,v^2_k,z^2_k,y^2_{kj})$ a feasible solution of model~\eqref{mod:final}. Since constraints \eqref{mod:final:dualcriterios}, \eqref{mod:final:dualescenarios} and \eqref{mod:final:dom1} are included in model~\eqref{mod:final}, $(z^2_k,y^2_{kj})$ is feasible for the model included in the RHS of constraint \eqref{mod:mininside:min} and therefore greater than or equal to the minimum of that model, verifying:
		$$z^2 + v^2_k \ge z^2_k + \sum_{j=1}^J \frac{\pi_j}{\beta}y^2_{kj} \ge \min \left\lbrace z_k + \sum_{j=1}^J \frac{\pi_j}{\beta}y_{kj} \right\rbrace$$
		and so, feasible for model~\eqref{mod:mininside}.
	\end{enumerate}
\end{proof}

Finally after proving the validity of model~\eqref{mod:final} it is possible to let $x \in X$ free, with the purpose of finding the one minimizing the function $h_r^\beta(x)$:
\begin{alignat*}{3}
	\min_{z,v_k,z_k,y_{kj},x} \quad & z + \sum_k \frac{w_k}{r}v_k  & \\
	\text{s.t.} \quad  & z + v_k \ge z_k + \sum_{j=1}^J \frac{\pi_j}{\beta} y_{kj} \quad & \forall k  \\
	   	             & z_k + y_{kj} \ge f_k^j(x) \quad & \forall k,j    \\
	                   & y_{kj} \ge 0 \quad & \forall k,j    \\
	                   & z_k \text{ free}, v_k \ge 0 \quad & \forall k   \\
	                   & z \text{ free}	  & \\
	                   & x \in X &
\end{alignat*}





\section{Knapsack problem}\label{sec:knap}
The multiobjective stochastic knapsack problem is used to illustrate the usefulness of the previously defined concept.
\begin{definition}[Multiobjective stochastic knapsack problem]
	Let $I$ be a collection of objects with weights $v_i$, which can be selected as members of a knapsack with maximum weight $V$. There is a set of scenarios $J$, each of them with probability $\pi_j$, and a set of criteria $K$, with importances $w_k$. For every pair of scenario-criterion, there is a benefit associated with selecting object $i$, denoted by $b^i_{jk}$. Which objects should be taken in order to maximize benefit?
\end{definition}

The above problem differs with the well-known knapsack problem in that there is stochasticity and multiple objectives to be maximized.

The following MSP model can be adapted to analyze the problem. Note that to preserve the sense of the optimization, rather than to maximize the benefits of the carried objects, it will be minimized the value of the objects not chosen.
\begin{equation}\label{model:knapMSP}
\begin{aligned} 
	\min_{x_i}   \quad& \left\lbrace f_k^j(\mathbf{x}):= \sum_i \left(1-x_i\right)b^i_{kj} \right\rbrace   \\
	\text{s.t.}\quad&  \sum_i v_i x_i \le V \quad & \forall i    \\
	                & x_i \in \{0,1\} \quad & \forall i 
\end{aligned}
\end{equation}

When using the methodology developed in the previous sections, problem~\eqref{model:knapMSP} is transformed into the following mixed-integer linear programming model:

\begin{equation}\tag{MSP}\label{model:MSP}
\begin{aligned}
	\min_{z,v_k,z_k,y_{kj},x_i} \quad & z + \sum_k \frac{w_k}{r}v_k  & \\
	\text{s.t.} \quad  & z + v_k \ge z_k + \sum_{j=1}^J \frac{\pi_j}{\beta} y_{kj} \quad & \forall k  \\
	& z_k + y_{kj} \ge \sum_i \left(1-x_i\right)b^i_{kj} \quad & \forall k,j    \\
	& \sum_i v_i x_i \le V \quad & \forall i    \\
	& y_{kj} \ge 0 \quad & \forall k,j    \\
	& x_i \in \{0,1\} \quad & \forall i    \\
	& z_k \text{ free}, v_k \ge 0 \quad & \forall k   \\
	& z \text{ free} &
\end{aligned}
\end{equation}

For given $r,\beta\in(0,1]$, model~\eqref{model:MSP} obtains the $\mathbf{x}^*$ minimizing the $r$-OWA of the $\beta$-averages. In order to illustrate the benefits of using model~\eqref{model:MSP}, a naive way of solving problem~\eqref{model:knapMSP} is considered:

\begin{equation*}\tag{MIP}\label{model:MIP}
\begin{aligned}
	\min_{x_i}   \quad& \sum_{k,j} w_k\pi_j \sum_i \left(1-x_i\right)b^i_{kj}  \\
	\text{s.t.}\quad&  \sum_i v_i x_i \le V \quad & \forall i    \\
	                & x_i \in \{0,1\} \quad & \forall i 
\end{aligned}
\end{equation*}

Hence model~\eqref{model:MIP} computes the average of the $f_k^j$, using the importances of the criteria and the probability of the scenarios. It is clear that for ``average'' criteria-scenarios $\xmip$, the optimal solution of model~\eqref{model:MIP}, outperforms $\xmsp$, the optimal solution of model~\eqref{model:MSP}. Conversely $\xmsp$ will improve $\xmip$ in unfavourable criteria-scenarios, as expected of a risk-averse solution.

\subsection{Computational experiments}
The following sections will show computational experiments, for different values of $r$ and $\beta$ and different number of objects, scenarios and criteria. Algorithm~\ref{alg:randominst} shows how the random instances are created, given a number of objects, scenarios and criteria.

\begin{algorithm}[ht!]
	\caption{Generating random data, with $\mathcal{U}(a,b)$ the uniform distribution in $[a,b]$}
	\label{alg:randominst}
	\begin{algorithmic}[1]
		\Function{randomInstance}{$|I|,|J|,|K|$}
		\State $p\gets \mathcal{U}(0.25,0.75)$\Comment{how many objects on average will fit in the knapsack}
		\State $W \gets \frac1p$\Comment{average weight of objects}
		\For{$i \in I$}
			\State $w_i\gets \mathcal{U}(0.5W,1.5W)$ \Comment{weight of each object}
			\For{$j,k \in J\times K$}
				\State $b^i_{kj}\gets \mathcal{U}(0,1)$ \Comment{value of each object}
			\EndFor
		\EndFor
		\EndFunction
	\end{algorithmic}
\end{algorithm}

For each of the solved instances it will be recorded:
\begin{itemize}
	\item $t_\text{MSP},t_\text{MIP}$: Solution time in seconds of models~\eqref{model:MSP} and \eqref{model:MIP}. With them the following value is calculated:
	$$ \Delta_\text{time} := \frac{t_\text{MSP}}{t_\text{MIP}}
		              \qquad \emph{(time penalty factor)}$$
	$\Delta_\text{time}$, the time penalty factor, indicates the increase of computing time when solving model~\eqref{model:MSP} rather than model~\eqref{model:MIP}.
	\item $\zmsp,\zmip$: Optimal values of the models. 
	\item $f_\text{MSP} \left(\xmip\right), f_\text{MIP} \left(\xmsp\right)$: Objective value of $\xmip$ in model~\eqref{model:MSP} and vice versa.
	\item To grasp the difference between the MSP and the naive approach, the following will be calculated:
	\begin{align*}
		\Delta_\text{avg}  & := 100\frac{f_\text{MIP} \left(\xmsp\right) -\zmip}{\zmip} 
		              \qquad \emph{(deteriorating rate)} \\
		\Delta_\text{tail} & := 100\frac{f_\text{MSP} \left(\xmip\right) -\zmsp}{f_\text{MSP} \left(\xmip\right)}  \qquad \emph{(improvement rate)} \\
	\end{align*}
	These quantities reflect what is the effect of making decision $\xmsp$ instead of $\xmip$.
	Large values of $\Delta_\text{avg}$ indicate high penalties for making decision $\xmsp$ instead of $\xmip$ in average scenarios-criteria. Similarly, the larger $\Delta_\text{tail}$ is, the higher benefit is obtained from making decision $\xmsp$ in tail events.
	They will be called \emph{deteriorating rate} and \emph{improvement rate}
\end{itemize}

Models are solved in GAMS 26.1.0 with solver IBM ILOG CPLEX Cplex 12.8.0.0, using a personal computer with an Intel Core i7 processor and 16Gb RAM.

\paragraph{Experiment 1}\todonum{set5}
First experiment will consist on a full factorial design, in which the values of $|I|,|J|,|K|,r,\beta$ fall in these sets:
\begin{itemize}
	\item $|I| \in \{50,100,200\}$
	\item $|J| \in \{5,25,100\}$
	\item $|K| \in \{3,6,9\}$
	\item $r \in \{0.33,0.5,0.67\}$
	\item $\beta \in \{0.05,0.1,0.5\}$
\end{itemize}
For each tuple $(I,J,K)$ random data will be generated, using algorithm~\ref{alg:randominst}, which will then be solved for every pair $(r,\beta)$. All criteria and scenarios are given same importance and probabilities. That is, $w_k = \frac1{|K|}$, $\pi_j = \frac1{|J|}$. Time limit was set in two hours by instance, in which all but three of the $3^5 = 243$ configurations were solved to optimality.

\paragraph{Experiment 2}\todonum{set6}
For the next experiment 100 random instances will be created, keeping the values of $|I|,|J|,|K|,r,\beta$ constant and equal to the median value of the previous experiment. That is, $|I|=100,|J|=25,|K|=6,r=0.5,\beta=0.1$. All criteria and scenarios are given same importance and probabilities. All 100 instances were solved to optimality.

\subsection{Results}
\paragraph{Experiment 1}
Table~\ref{tab:set5:alldata} shows for each of the 243 instances the solution times of the MSP and the MIP models, and the deteriorating and improvement rates of using the MSP solution instead of the MIP solutions (measured in deviation to MIP solution).

Table~\ref{tab:set5:corrmatrix} shows the correlations between times and rates with the parameters of the instance. It can be seen how the MSP solution has a higher impact when fewer scenarios are considered. In addition to that, it can be appreciated that the MSP solution times decrease when $\beta$ increase, that is, when more scenarios are included in the $\beta$-average computation.

\begin{table}[ht!]
	\centering
	\caption{Correlations}
	\label{tab:set5:corrmatrix}
	\begin{tabular}{lrrrrr}\toprule
		{}                   &   |I| &   |J| &   |K| &   $r$ & $\beta$ \\
		\midrule
		$t_\text{MSP}$       &  0.34 &  0.09 & -0.11 & -0.05 & -0.19 \\
		$t_\text{MIP}$       &  0.51 &  0.18 & -0.14 & -0.03 & -0.07 \\
		$\Delta_\text{time}$ &  0.31 &  0.11 & -0.08 & -0.02 & -0.18 \\
		$\Delta_\text{avg}$  & -0.05 & -0.57 & -0.28 & -0.09 & -0.36 \\
		$\Delta_\text{tail}$ & -0.07 & -0.56 & -0.18 & -0.21 & -0.50 \\
		\bottomrule
	\end{tabular}
\end{table}

This appreciation is confirmed by Table~\ref{tab:set5:groupedB}, in which it can be seen that the median \emph{time penalty factor} (how many more times does it take to solve the MSP model than the MIP one) is much smaller when $\beta=0.5$ than when $\beta=0.05$.

\begin{table}[ht!]
	\centering
	\caption{Increase on computing times and MSP solution times, grouped by $\beta$}
	\label{tab:set5:groupedB}
	\resizebox{\linewidth}{!}{
	\begin{tabular}{l||rrrrr|rrrrr}\toprule
		{} & \multicolumn{5}{c|}{$\Delta_\text{time}$} & \multicolumn{5}{c}{$t_\text{MSP}$} \\
		$\beta$ &     min &     mean & median &       max &      std &       min &    mean & median &      max &      std \\
		\midrule
		0.05 &    0.94 &  3188.96 &  32.77 &  50473.04 &  9472.55 &      0.12 &  659.49 &   6.32 &  7222.95 &  1787.07 \\
		0.10 &    0.98 &  1002.35 &  11.09 &  20192.48 &  3245.85 &      0.12 &  212.47 &   2.23 &  4765.42 &   728.49 \\
		0.50 &    1.06 &    19.14 &   3.75 &    414.29 &    55.51 &      0.13 &    3.49 &   0.67 &    62.14 &     9.05 \\
		\bottomrule
	\end{tabular}}
\end{table}

Solution times of the MSP model are alarmingly high for some instances, due to the fact that the admissible integrality gap has been set to zero. If that is relaxed, it can be seen that all of the 243 instances reach an integrality gap smaller than 5\% in under 3 seconds, 2\% in under 5 seconds and 1\% in under 88 seconds.

Table~\ref{tab:set5:groupedRB} groups instances by $r$ and $\beta$, and shows the deteriorating and improvement rates. It can be seen that the improvement rate (in the tail) is generally higher than the deteriorating rate (in the average), especially in cases with small $r$ and $\beta$. 

\begin{table}[ht!]
	\centering
	\caption{Values of $\Delta_\text{avg}$ and $\Delta_\text{tail}$, grouped by $r$ and $\beta$}
	\label{tab:set5:groupedRB}
	\resizebox{\columnwidth}{!}{%
	\begin{tabular}{ll||rrrrr|rrrrr}\toprule
		     &      & \multicolumn{5}{c|}{$\Delta_\text{avg}$} & \multicolumn{5}{c}{$\Delta_\text{tail}$} \\
		 $r$ & $\beta$ &   min &  mean & median &   max &   std &    min &  mean & median &   max &   std \\
		\midrule
		0.33 &    0.05 &  0.03 &  1.94 &   1.87 &  5.68 &  1.42 &   0.28 &  4.37 &   4.21 &  9.18 &  2.43 \\
		     &    0.10 &  0.02 &  1.70 &   1.61 &  5.68 &  1.44 &   0.18 &  3.54 &   2.85 &  9.18 &  2.42 \\
		     &    0.50 &  0.00 &  0.93 &   0.52 &  4.46 &  1.08 &   0.00 &  1.57 &   0.92 &  4.99 &  1.46 \\
		0.50 &    0.05 &  0.03 &  1.87 &   1.90 &  4.30 &  1.30 &   0.29 &  3.58 &   3.30 &  6.73 &  1.89 \\
		     &    0.10 &  0.02 &  1.65 &   1.14 &  4.30 &  1.37 &   0.13 &  2.87 &   2.47 &  6.59 &  1.86 \\
		     &    0.50 &  0.00 &  0.72 &   0.54 &  3.51 &  0.75 &   0.00 &  1.07 &   0.79 &  3.79 &  1.01 \\
		0.67 &    0.05 &  0.03 &  1.64 &   1.17 &  3.93 &  1.24 &   0.32 &  3.04 &   3.06 &  6.15 &  1.62 \\
		     &    0.10 &  0.01 &  1.50 &   1.10 &  3.93 &  1.31 &   0.12 &  2.43 &   2.02 &  5.84 &  1.58 \\
		     &    0.50 &  0.00 &  0.60 &   0.50 &  3.16 &  0.66 &   0.00 &  0.80 &   0.59 &  3.64 &  0.81 \\
		\bottomrule
	\end{tabular}}
\end{table}

This claim is also supported with Figure~\ref{fig:clustersSet5}, where each of the 243 instances is shown according to the values of $\Delta_\text{avg}$ and $\Delta_\text{tail}$, and grouped by the values of $(r,\beta)$. Almost all of the instances ara above the imaginary line $\Delta_\text{avg} = \Delta_\text{tail}$, which shows that considering the MSP solution improves in the tail more than it loses in the average situations. In addition to that, it can be seen that the largest improvements in the tail are on instances with $\beta = 0.05$ (one of the usual values taken for CVaR), and especially with the smallest values of $r$. When $r$ and $\beta$ grow the differences between the MIP and MSP solutions are reduced.

\begin{figure}[ht!]
	\centering
	\includegraphics[width=0.7\textwidth]{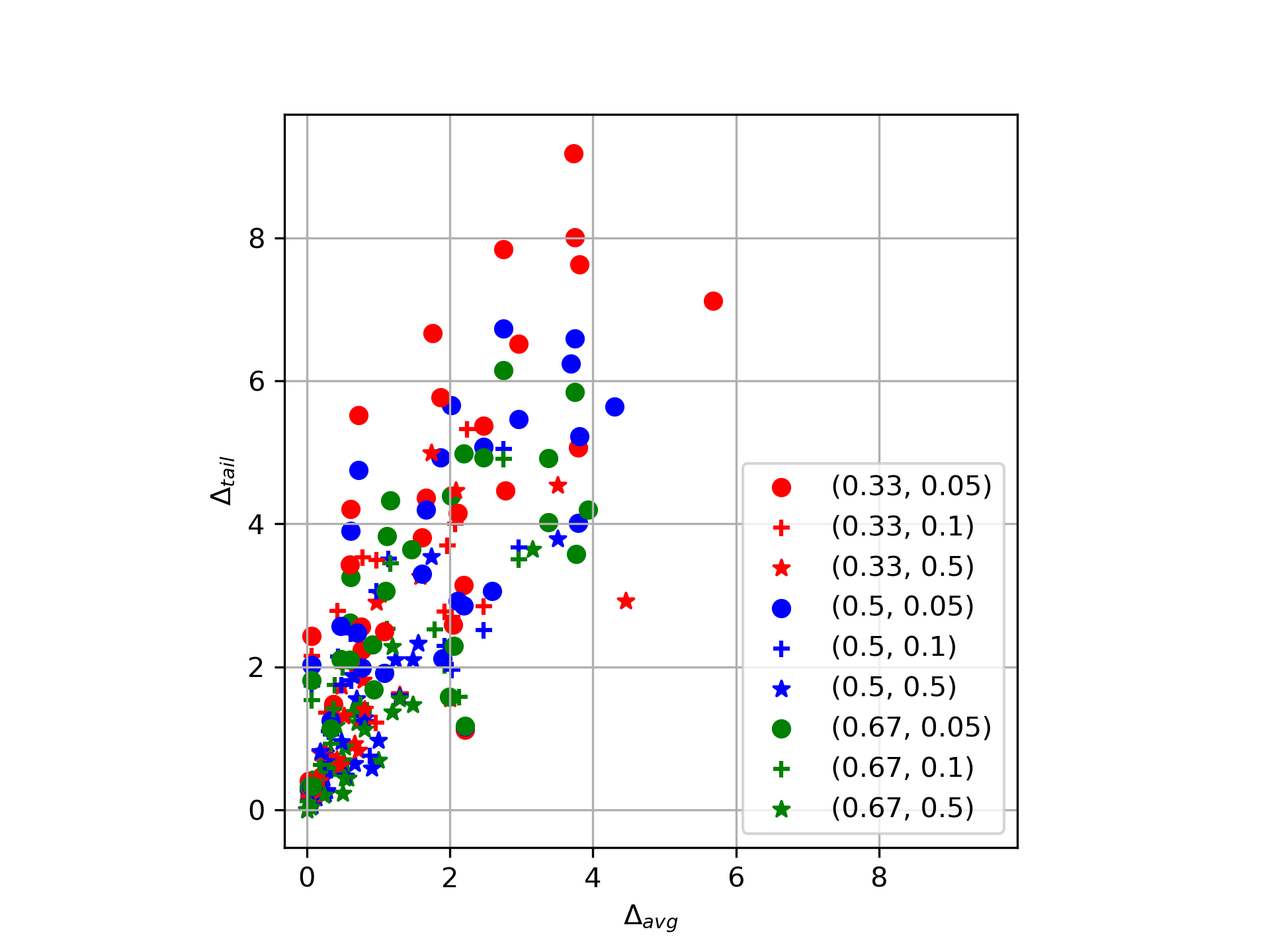}
	\caption{Values $\Delta_\text{avg}$ and $\Delta_\text{tail}$ for each of the 243 instances, grouped by values of $(r,\beta)$}
	\label{fig:clustersSet5}
\end{figure}

\paragraph{Experiment 2}
Table~\ref{tab:set6:alldata} contains the results for each of the 100 instances, all of them with constant parameters $|I|=100,|J|=25,|K|=6,r=0.5,\beta=0.1$.

Table~\ref{tab:set6:Summary} contains a summary of the results, where it is again seen that the improvements in the tail are better than the loses in the average situations. Although single instances might take a long computing time, the median MSP solution time (3.74s) is definitely satisfactory. It is worth mentioning that the models were implemented without providing any extra bounds or known cuts that could reduce solution times.

\begin{table}[ht!]
	\centering
	\caption{Summary of set 6}
	\label{tab:set6:Summary}
	\begin{tabular}{lrrrrr}\toprule
		{} &  $t_\text{MSP}$ &  $t_\text{MIP}$ &  $\Delta_\text{time}$ &   $\Delta_\text{avg}$ &  $\Delta_\text{tail}$ \\
		\midrule
		mean  &      16.98 &       0.20 &    91.31 &    2.03 &    3.09 \\
		std   &      46.57 &       0.03 &   254.68 &    1.12 &    1.49 \\
		min   &       0.53 &       0.14 &     2.81 &    0.16 &    0.86 \\
		25\%   &       1.37 &       0.17 &     6.73 &    1.18 &    2.09 \\
		50\%   &       3.74 &       0.19 &    19.72 &    1.93 &    2.81 \\
		75\%   &      15.50 &       0.21 &    86.19 &    2.52 &    3.51 \\
		max   &     404.70 &       0.34 &  2175.82 &    5.67 &    8.57 \\
		\bottomrule
	\end{tabular}
\end{table}

Finally, figure~\ref{fig:unainstancia} shows the values of $f_k^j(x)$, where $x=\xmip$ in blue squares and $x=\xmsp$ in orange circles, for just one of the created instances. It can be appreciated how $\xmip$ performs better than $\xmsp$ in average criteria-scenarios, but $\xmsp$ is better with unfavourable situations. 

\begin{figure}[ht!]
	\centering
	\begin{subfigure}{.5\textwidth}
		\centering
		\includegraphics[width=\linewidth]{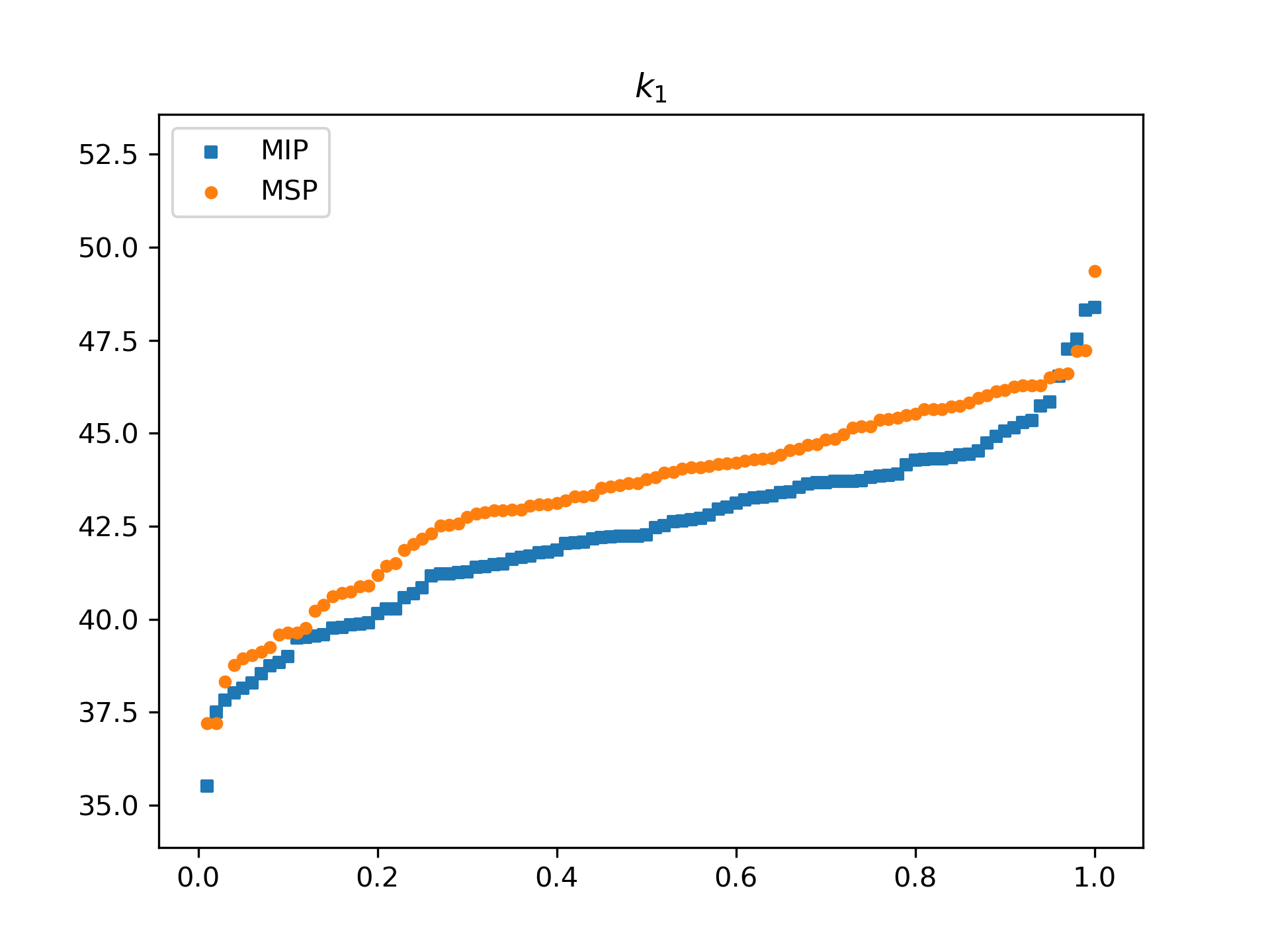}
	\end{subfigure}%
	\begin{subfigure}{.5\textwidth}
		\centering
		\includegraphics[width=\linewidth]{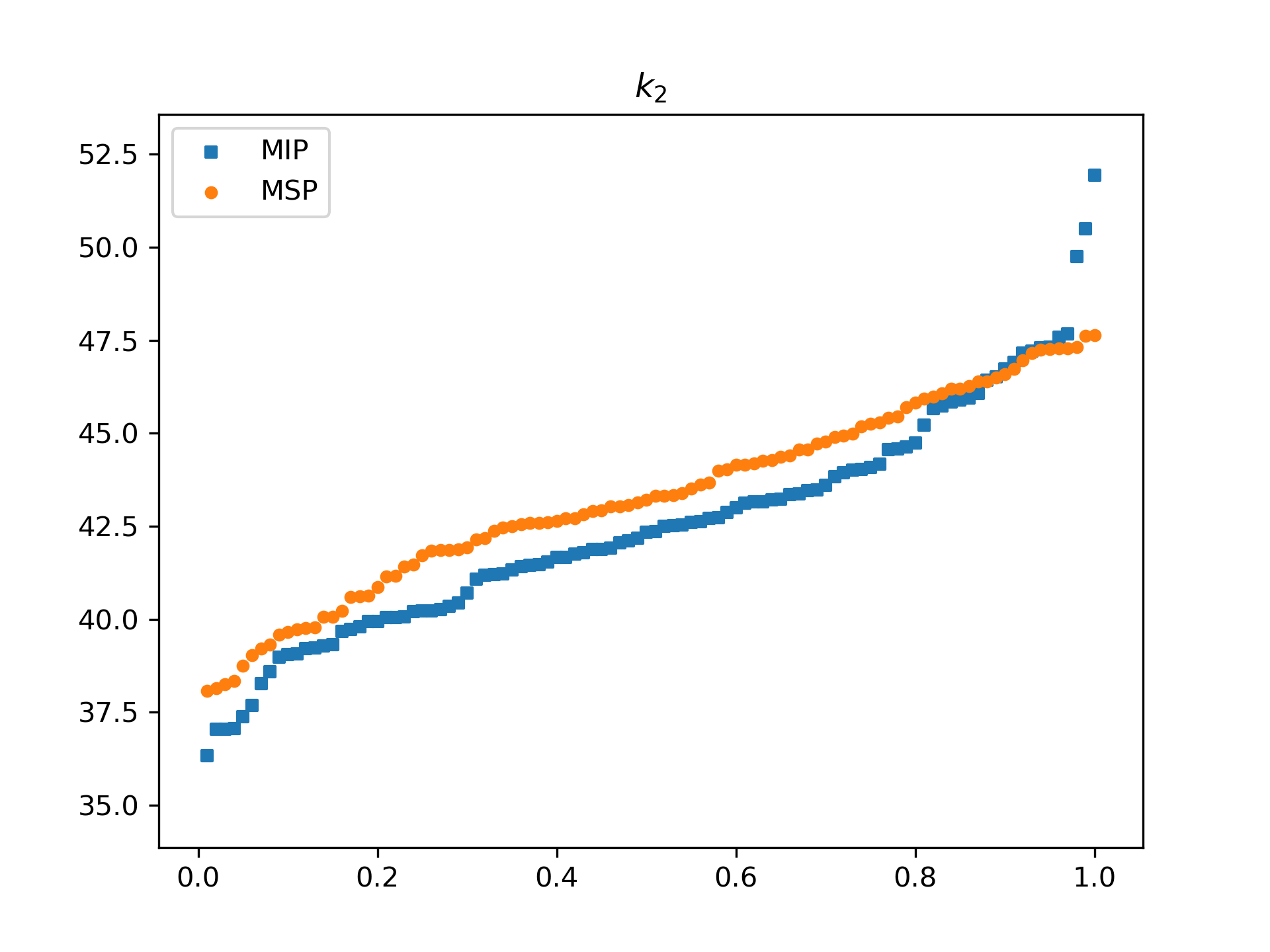}
	\end{subfigure}
	\vskip\baselineskip
	\begin{subfigure}{.5\textwidth}
		\centering
		\includegraphics[width=\linewidth]{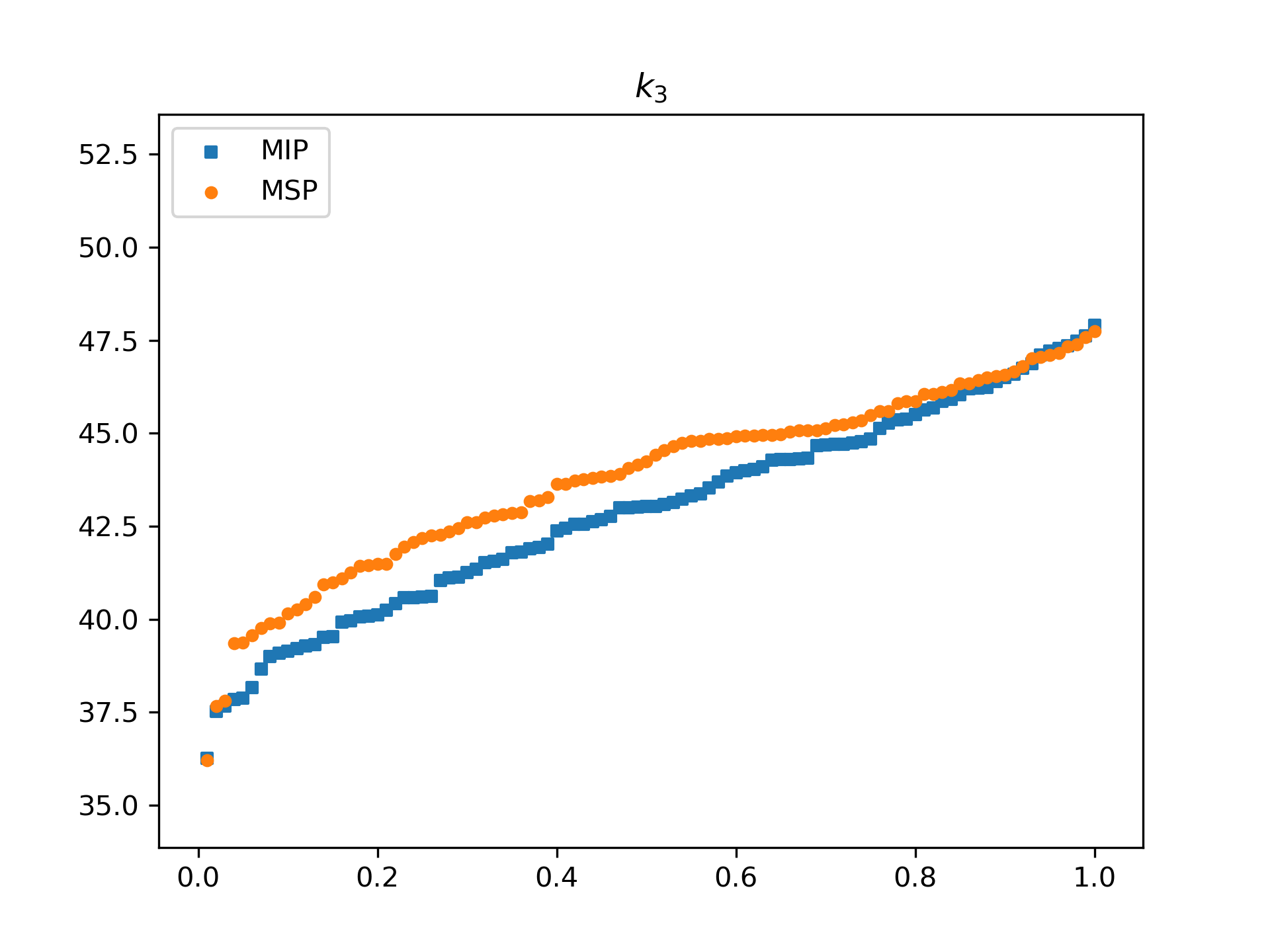}
	\end{subfigure}
	\caption{Single instance with 100 scenarios and 3 criteria. For each $k$, sorted values of $f_k^j(x)$, where $x=\xmip$ in blue squares and $x=\xmsp$ in orange circles}
	\label{fig:unainstancia}
\end{figure}

\section{Conclusions}\label{sec:conclusions}
In this paper a new concept of solution has been proposed for Multiobjective Stochastic Programming problems, focused on risk-aversion. As such, this concept can be particularly useful in real-life situations where there exists a great concern with respect to unfavourable situations, such as emergency management.

The solution concept is supported by an efficient way to compute it by a Mathematical Programming problem. This model is linear provided that the underlying problem can be linearly representable. Numerical experiments have been conducted for validating this approach, solving a multiobjective stochastic knapsack problem.

The research has also shown that the improvements in the tail (unfavourable situations) are consistently higher than loses on average situations, especially when small values of the parameters $\beta$ and $r$ are chosen. These differences, although clearly noticeable, are not as high as one could expect. This is possibly due to the randomness of the data. It is reasonable to assume that in actual real-life problems there are choices that are more conservative for every scenario and criterion, and thus being preferable for risk-aversion attitudes.

Results showed that there is a clear increase in computational time; however this is arguably acceptable as a price to pay for being risk-averse. Furthermore, this could also be due to the random nature of the data. Nevertheless, it was also shown that allowing for even rather small integrality gaps (1\%) leads to drastic improvement on the computing times.

\section*{Acknowledgements}
This work has been supported by the UCM-Santander grant CT27/16-CT28/16, the Government of Spain grants MTM2016-74983-C02-01 and MTM2015-65803-R, and grant H2020 MSCA-RISE 691161 (GEO-SAFE).

\bibliography{inputs/biblioMSP}

\clearpage
\appendix
\section{Extra figures and tables}
\begin{table}[ht!]
	\centering
	\caption{Values of alternative 2 by scenario ($j$) and criteria ($k$)}
	\label{tab:inst1alt2}
	\resizebox{\textwidth}{!}{
	\begin{tabular}{rrr|cccccc|}
		& & & \multicolumn{6}{c|}{criteria} \\
		& & & $w_1 = 0.20$ & $w_2 = 0.10$ & $w_3 = 0.20$ & $w_4 = 0.25$ & $w_5 = 0.15$ & $w_6 = 0.10$ \\
		& & & $k_1$ & $k_2$ & $k_3$ & $k_4$ & $k_5$ & $k_6$ \\ \hline
		\parbox[t]{4mm}{\multirow{5}{*}{\rotatebox[origin=c]{90}{scenarios}}}
		& $\pi_1 = 0.15$ & $j_1$ &  0.40  &  0.58  &  0.39  &  0.45  &  0.54  &  0.18  \\
		& $\pi_2 = 0.20$ & $j_2$ &  0.68  &  0.74  &  0.70  &  0.15  &  0.54  &  0.72  \\
		& $\pi_3 = 0.30$ & $j_3$ &  0.93  &  0.52  &  0.23  &  0.82  &  0.21  &  0.03  \\
		& $\pi_4 = 0.25$ & $j_4$ &  0.37  &  0.85  &  0.07  &  0.42  &  0.52  &  0.22  \\
		& $\pi_5 = 0.10$ & $j_5$ &  0.92  &  0.13  &  0.71  &  0.39  &  0.90  &  0.87  \\ \hline 
		\multicolumn{3}{c|}{$\beta$-average, $\beta = 0.30$} 
		            & 0.930 & 0.832 & 0.703 & 0.820 & 0.660 & 0.770 \\ \hline
		\multicolumn{3}{c|}{$r$-OWA, $r = 0.17$} & \multicolumn{6}{c|}{0.930} \\ \hline
	\end{tabular}}
\end{table}

\begin{table}[ht!]
	\centering
	\caption{Values of alternative 3 by scenario ($j$) and criteria ($k$)}
	\label{tab:inst1alt3}
	\resizebox{\textwidth}{!}{
	\begin{tabular}{rrr|cccccc|}
		& & & \multicolumn{6}{c|}{criteria} \\
		& & & $w_1 = 0.20$ & $w_2 = 0.10$ & $w_3 = 0.20$ & $w_4 = 0.25$ & $w_5 = 0.15$ & $w_6 = 0.10$ \\
		& & & $k_1$ & $k_2$ & $k_3$ & $k_4$ & $k_5$ & $k_6$ \\ \hline
		\parbox[t]{4mm}{\multirow{5}{*}{\rotatebox[origin=c]{90}{scenarios}}}
		& $\pi_1 = 0.15$ & $j_1$ &  0.80  &  0.90  &  0.61  &  0.28  &  0.94  &  0.09  \\
		& $\pi_2 = 0.20$ & $j_2$ &  0.29  &  0.48  &  0.26  &  0.23  &  0.21  &  0.07  \\
		& $\pi_3 = 0.30$ & $j_3$ &  0.73  &  0.65  &  0.32  &  0.56  &  0.95  &  0.65  \\
		& $\pi_4 = 0.25$ & $j_4$ &  0.58  &  0.39  &  0.21  &  0.66  &  0.70  &  0.93  \\
		& $\pi_5 = 0.10$ & $j_5$ &  0.73  &  0.22  &  0.33  &  0.31  &  0.32  &  0.38  \\ \hline 
		\multicolumn{3}{c|}{$\beta$-average, $\beta = 0.30$} 
		            & 0.765 & 0.775 & 0.468 & 0.643 & 0.950 & 0.883 \\ \hline
		\multicolumn{3}{c|}{$r$-OWA, $r = 0.17$} & \multicolumn{6}{c|}{0.943} \\ \hline
	\end{tabular}}
\end{table}

\begin{table}[ht!]
	\centering
	\caption{Values of alternative 4 by scenario ($j$) and criteria ($k$)}
	\label{tab:inst1alt4}
	\resizebox{\textwidth}{!}{
	\begin{tabular}{rrr|cccccc|}
		& & & \multicolumn{6}{c|}{criteria} \\
		& & & $w_1 = 0.20$ & $w_2 = 0.10$ & $w_3 = 0.20$ & $w_4 = 0.25$ & $w_5 = 0.15$ & $w_6 = 0.10$ \\
		& & & $k_1$ & $k_2$ & $k_3$ & $k_4$ & $k_5$ & $k_6$ \\ \hline
		\parbox[t]{4mm}{\multirow{5}{*}{\rotatebox[origin=c]{90}{scenarios}}}
		& $\pi_1 = 0.15$ & $j_1$ &  0.30  &  0.52  &  0.12  &  0.68  &  0.46  &  0.73  \\
		& $\pi_2 = 0.20$ & $j_2$ &  1.00  &  0.57  &  0.46  &  0.82  &  0.90  &  0.72  \\
		& $\pi_3 = 0.30$ & $j_3$ &  0.18  &  0.76  &  0.30  &  0.34  &  0.54  &  0.99  \\
		& $\pi_4 = 0.25$ & $j_4$ &  0.53  &  0.21  &  0.13  &  0.12  &  0.66  &  0.86  \\
		& $\pi_5 = 0.10$ & $j_5$ &  0.98  &  0.46  &  0.50  &  0.29  &  0.27  &  0.40  \\ \hline 
		\multicolumn{3}{c|}{$\beta$-average, $\beta = 0.30$} 
		            & 0.993 & 0.760 & 0.473 & 0.773 & 0.820 & 0.990 \\ \hline
		\multicolumn{3}{c|}{$r$-OWA, $r = 0.17$} & \multicolumn{6}{c|}{0.993} \\ \hline
	\end{tabular}}
\end{table}



\newgeometry{bottom=1cm,top=1cm}
\begingroup
\thispagestyle{empty}
\renewcommand{\arraystretch}{1.3}
\setlength{\tabcolsep}{4pt}
\begin{landscape}
\begin{table}[ht!]
	\centering
	\caption{All instances of first experiment. The three instances with 200 objects, 100 scenarios, 6 criteria and $\beta=0.05$ did not reach the optimal solution in 2 hours. The integrality gaps of the solution shown are $0.31\%, 0.24\%$ and $0.19\%$ for $r=0.33, 0.5$ and $0.67$ respectively}
	\label{tab:set5:alldata}
	\resizebox{\columnwidth}{!}{%
    \begin{tabular}{rrr||rrrr|rrrr|rrrr||rrrr|rrrr|rrrr||rrrr|rrrr|rrrr||}
      \multicolumn{3}{r||}{$\beta\to$} & \multicolumn{12}{c||}{0.05}                                 & \multicolumn{12}{c||}{0.1}                                  & \multicolumn{12}{c||}{0.5} \\ \hline
      \multicolumn{3}{r||}{$r\to$} & \multicolumn{4}{c|}{0.33} & \multicolumn{4}{c|}{0.5} & \multicolumn{4}{c||}{0.67} & \multicolumn{4}{c|}{0.33} & \multicolumn{4}{c|}{0.5} & \multicolumn{4}{c||}{0.67} & \multicolumn{4}{c|}{0.33} & \multicolumn{4}{c|}{0.5} & \multicolumn{4}{c||}{0.67} \\ \hline
    \multicolumn{1}{l}{|I|} & \multicolumn{1}{l}{|J|} & \multicolumn{1}{l||}{|K|} & \multicolumn{1}{l}{$t_\text{MSP}$} & \multicolumn{1}{l}{$t_\text{MIP}$} & \multicolumn{1}{l}{$\Delta_\text{avg}$} & \multicolumn{1}{l|}{$\Delta_\text{tail}$} & \multicolumn{1}{l}{$t_\text{MSP}$} & \multicolumn{1}{l}{$t_\text{MIP}$} & \multicolumn{1}{l}{$\Delta_\text{avg}$} & \multicolumn{1}{l|}{$\Delta_\text{tail}$} & \multicolumn{1}{l}{$t_\text{MSP}$} & \multicolumn{1}{l}{$t_\text{MIP}$} & \multicolumn{1}{l}{$\Delta_\text{avg}$} & \multicolumn{1}{l||}{$\Delta_\text{tail}$} & \multicolumn{1}{l}{$t_\text{MSP}$} & \multicolumn{1}{l}{$t_\text{MIP}$} & \multicolumn{1}{l}{$\Delta_\text{avg}$} & \multicolumn{1}{l|}{$\Delta_\text{tail}$} & \multicolumn{1}{l}{$t_\text{MSP}$} & \multicolumn{1}{l}{$t_\text{MIP}$} & \multicolumn{1}{l}{$\Delta_\text{avg}$} & \multicolumn{1}{l|}{$\Delta_\text{tail}$} & \multicolumn{1}{l}{$t_\text{MSP}$} & \multicolumn{1}{l}{$t_\text{MIP}$} & \multicolumn{1}{l}{$\Delta_\text{avg}$} & \multicolumn{1}{l||}{$\Delta_\text{tail}$} & \multicolumn{1}{l}{$t_\text{MSP}$} & \multicolumn{1}{l}{$t_\text{MIP}$} & \multicolumn{1}{l}{$\Delta_\text{avg}$} & \multicolumn{1}{l|}{$\Delta_\text{tail}$} & \multicolumn{1}{l}{$t_\text{MSP}$} & \multicolumn{1}{l}{$t_\text{MIP}$} & \multicolumn{1}{l}{$\Delta_\text{avg}$} & \multicolumn{1}{l|}{$\Delta_\text{tail}$} & \multicolumn{1}{l}{$t_\text{MSP}$} & \multicolumn{1}{l}{$t_\text{MIP}$} & \multicolumn{1}{l}{$\Delta_\text{avg}$} & \multicolumn{1}{l||}{$\Delta_\text{tail}$} \\ \hline
		\hline
		50  & 5   & 3 &     0.12 &  0.13 &  3.75 &   8.01 &     0.15 &  0.12 &  3.75 &   8.01 &   0.18 &  0.14 &  3.51 &   4.54 &     0.14 &  0.14 &  3.75 &   6.59 &     0.12 &  0.12 &  3.75 &   6.59 &   0.20 &  0.17 &  3.51 &   3.79 &     0.12 &  0.12 &  3.75 &   5.84 &     0.12 &  0.12 &  3.75 &   5.84 &  0.13 &  0.12 &  3.16 &   3.64 \\
		    &     & 6 &     0.22 &  0.11 &  3.79 &   5.07 &     0.25 &  0.11 &  3.79 &   5.07 &   0.18 &  0.11 &  1.03 &   3.01 &     0.30 &  0.13 &  3.79 &   4.01 &     0.25 &  0.13 &  3.79 &   4.01 &   0.23 &  0.15 &  1.48 &   2.10 &     0.23 &  0.16 &  3.77 &   3.58 &     0.23 &  0.14 &  3.77 &   3.58 &  0.18 &  0.17 &  1.48 &   1.47 \\
		    &     & 9 &     0.28 &  0.14 &  1.76 &   6.67 &     0.36 &  0.14 &  1.76 &   6.67 &   0.20 &  0.15 &  1.58 &   3.26 &     0.22 &  0.14 &  2.02 &   5.66 &     0.21 &  0.14 &  2.02 &   5.66 &   0.22 &  0.14 &  1.24 &   2.10 &     0.20 &  0.15 &  2.02 &   4.39 &     0.22 &  0.13 &  2.02 &   4.39 &  0.25 &  0.15 &  1.20 &   1.37 \\
		    & 25  & 3 &     0.76 &  0.18 &  2.47 &   5.37 &     0.66 &  0.17 &  2.47 &   2.85 &   0.26 &  0.17 &  2.00 &   1.55 &     0.64 &  0.18 &  2.47 &   5.08 &     0.62 &  0.16 &  2.47 &   2.51 &   0.25 &  0.15 &  1.00 &   0.97 &     0.60 &  0.17 &  2.47 &   4.93 &     0.51 &  0.16 &  1.79 &   2.52 &  0.26 &  0.14 &  1.00 &   0.69 \\
		    &     & 6 &     1.20 &  0.16 &  1.87 &   5.77 &     0.91 &  0.29 &  1.96 &   3.70 &   0.49 &  0.18 &  0.79 &   1.81 &     1.15 &  0.18 &  1.87 &   4.93 &     0.74 &  0.18 &  1.14 &   3.51 &   0.41 &  0.18 &  0.70 &   1.55 &     0.93 &  0.16 &  1.17 &   4.33 &     0.78 &  0.18 &  1.17 &   3.45 &  0.38 &  0.16 &  0.71 &   1.22 \\
		    &     & 9 &     0.69 &  0.16 &  0.61 &   4.21 &     0.78 &  0.16 &  0.43 &   2.78 &   0.57 &  0.16 &  0.67 &   0.92 &     0.52 &  0.15 &  0.61 &   3.90 &     0.96 &  0.16 &  0.44 &   2.14 &   0.61 &  0.16 &  0.67 &   0.65 &     0.69 &  0.15 &  0.61 &   3.25 &     1.02 &  0.18 &  0.39 &   1.75 &  0.47 &  0.18 &  0.51 &   0.68 \\
		    & 100 & 3 &     1.15 &  0.15 &  0.07 &   2.43 &     0.78 &  0.15 &  0.07 &   2.15 &   0.44 &  0.14 &  0.07 &   0.16 &     1.07 &  0.14 &  0.07 &   2.02 &     0.83 &  0.19 &  0.07 &   1.74 &   0.34 &  0.14 &  0.07 &   0.16 &     1.14 &  0.15 &  0.07 &   1.81 &     0.85 &  0.16 &  0.07 &   1.53 &  0.36 &  0.14 &  0.07 &   0.14 \\
		    &     & 6 &     2.51 &  0.20 &  0.77 &   2.24 &     4.45 &  0.22 &  0.78 &   1.19 &   3.52 &  0.20 &  0.23 &   0.47 &     2.62 &  0.25 &  0.77 &   1.99 &     5.09 &  0.18 &  0.31 &   1.10 &   5.45 &  0.20 &  0.19 &   0.29 &     2.70 &  0.22 &  0.77 &   1.38 &     3.31 &  0.19 &  0.47 &   0.63 &  5.59 &  0.19 &  0.23 &   0.27 \\
		    &     & 9 &     4.06 &  0.18 &  0.03 &   0.41 &     1.47 &  0.17 &  0.03 &   0.30 &   1.07 &  0.16 &  0.00 &   0.00 &     2.75 &  0.17 &  0.03 &   0.29 &     1.12 &  0.16 &  0.03 &   0.30 &   1.11 &  0.16 &  0.00 &   0.00 &     2.06 &  0.19 &  0.03 &   0.32 &     1.10 &  0.16 &  0.03 &   0.18 &  1.16 &  0.17 &  0.00 &   0.00 \\
		100 & 5   & 3 &     1.24 &  0.26 &  3.81 &   7.63 &     1.29 &  0.20 &  3.81 &   7.63 &   0.37 &  0.22 &  2.08 &   4.47 &     0.72 &  0.18 &  3.81 &   5.22 &     0.66 &  0.18 &  3.81 &   5.22 &   0.32 &  0.27 &  1.56 &   2.33 &     0.74 &  0.19 &  3.38 &   4.02 &     1.13 &  0.23 &  3.38 &   4.02 &  0.28 &  0.20 &  0.63 &   1.36 \\
		    &     & 6 &     8.68 &  0.22 &  5.68 &   7.12 &     8.69 &  0.19 &  5.68 &   7.12 &   0.43 &  0.17 &  4.46 &   2.92 &     0.65 &  0.20 &  4.30 &   5.64 &     1.06 &  0.19 &  4.30 &   5.64 &   0.35 &  0.17 &  0.81 &   1.28 &     1.18 &  0.22 &  3.93 &   4.20 &     0.64 &  0.18 &  3.93 &   4.20 &  0.28 &  0.18 &  0.53 &   0.88 \\
		    &     & 9 &     3.31 &  0.18 &  2.19 &   3.14 &     3.26 &  0.20 &  2.19 &   3.14 &   0.67 &  0.18 &  0.97 &   2.90 &     1.04 &  0.20 &  2.19 &   2.86 &     0.96 &  0.20 &  2.19 &   2.86 &   0.23 &  0.14 &  0.64 &   1.88 &     1.02 &  0.17 &  0.92 &   2.31 &     0.90 &  0.17 &  0.92 &   2.31 &  0.24 &  0.18 &  0.41 &   1.18 \\
		    & 25  & 3 &    10.65 &  0.17 &  2.96 &   6.52 &     3.39 &  0.18 &  2.07 &   4.00 &   0.29 &  0.15 &  0.48 &   1.73 &     7.09 &  0.18 &  2.96 &   5.46 &     1.83 &  0.19 &  2.96 &   3.67 &   0.30 &  0.14 &  0.48 &   0.95 &     3.46 &  0.19 &  2.19 &   4.98 &     1.30 &  0.16 &  2.96 &   3.50 &  0.34 &  0.16 &  0.41 &   0.59 \\
		    &     & 6 &    32.12 &  0.20 &  2.78 &   4.47 &     9.18 &  0.19 &  0.78 &   3.53 &   0.44 &  0.18 &  0.52 &   1.31 &    26.53 &  0.18 &  2.59 &   3.06 &     3.64 &  0.22 &  0.61 &   2.47 &   0.32 &  0.15 &  0.26 &   0.79 &    12.77 &  0.17 &  0.60 &   2.62 &     0.90 &  0.17 &  0.50 &   2.00 &  0.41 &  0.17 &  0.26 &   0.65 \\
		    &     & 9 &     8.58 &  0.18 &  0.72 &   5.52 &     1.90 &  0.17 &  0.97 &   3.49 &   0.42 &  0.18 &  0.24 &   0.82 &     6.32 &  0.19 &  0.72 &   4.75 &     1.24 &  0.16 &  0.97 &   3.06 &   0.51 &  0.20 &  0.24 &   0.44 &     1.60 &  0.19 &  1.12 &   3.83 &     0.88 &  0.19 &  1.12 &   2.53 &  0.59 &  0.17 &  0.50 &   0.23 \\
		    & 100 & 3 &    51.23 &  0.22 &  2.21 &   1.12 &     1.67 &  0.21 &  0.27 &   1.36 &   0.82 &  0.18 &  0.09 &   0.25 &    22.70 &  0.21 &  2.21 &   1.16 &     1.22 &  0.19 &  0.34 &   1.05 &   0.81 &  0.18 &  0.05 &   0.17 &    18.75 &  0.16 &  2.21 &   1.17 &     0.84 &  0.22 &  0.34 &   0.92 &  0.75 &  0.18 &  0.05 &   0.13 \\
		    &     & 6 &    48.25 &  0.18 &  0.76 &   2.56 &    31.87 &  0.17 &  0.62 &   2.05 &  62.14 &  0.15 &  0.42 &   0.70 &    24.73 &  0.18 &  0.71 &   2.48 &    27.08 &  0.18 &  0.62 &   1.81 &  42.18 &  0.19 &  0.28 &   0.55 &    20.26 &  0.17 &  0.60 &   2.10 &    22.09 &  0.20 &  0.75 &   1.48 &  7.79 &  0.20 &  0.17 &   0.50 \\
		    &     & 9 &     2.16 &  0.19 &  0.37 &   1.48 &     3.34 &  0.18 &  0.29 &   0.77 &   1.84 &  0.17 &  0.18 &   0.41 &     1.80 &  0.17 &  0.34 &   1.25 &     2.87 &  0.20 &  0.28 &   0.69 &   2.22 &  0.19 &  0.26 &   0.22 &     1.67 &  0.18 &  0.34 &   1.14 &     2.77 &  0.19 &  0.20 &   0.63 &  3.09 &  0.16 &  0.08 &   0.13 \\
		200 & 5   & 3 &   146.24 &  0.23 &  1.61 &   3.81 &   140.12 &  0.20 &  1.61 &   3.81 &   7.71 &  0.23 &  1.30 &   1.61 &   151.22 &  0.21 &  1.61 &   3.30 &   135.09 &  0.24 &  1.61 &   3.30 &   4.60 &  0.21 &  1.30 &   1.58 &    83.44 &  0.22 &  1.10 &   3.06 &    89.20 &  0.21 &  1.10 &   3.06 &  4.21 &  0.22 &  1.30 &   1.55 \\
		    &     & 6 &    88.70 &  0.19 &  1.08 &   2.50 &    89.69 &  0.19 &  1.08 &   2.50 &   5.14 &  0.17 &  0.72 &   0.83 &    96.44 &  0.19 &  1.08 &   1.91 &    91.66 &  0.18 &  1.08 &   1.91 &   2.93 &  0.18 &  0.91 &   0.58 &    39.26 &  0.18 &  0.94 &   1.68 &    32.92 &  0.18 &  0.94 &   1.68 &  0.70 &  0.17 &  0.58 &   0.44 \\
		    &     & 9 &   468.37 &  0.15 &  3.73 &   9.18 &   484.89 &  0.14 &  3.73 &   9.18 &  29.46 &  0.16 &  1.74 &   4.99 &   304.04 &  0.16 &  3.69 &   6.24 &   305.90 &  0.16 &  3.69 &   6.24 &   2.71 &  0.16 &  1.74 &   3.54 &   110.03 &  0.15 &  3.38 &   4.92 &   107.34 &  0.14 &  3.38 &   4.92 &  0.91 &  0.17 &  1.20 &   2.28 \\
		    & 25  & 3 &  5629.58 &  0.33 &  2.75 &   7.84 &  4765.42 &  0.24 &  2.24 &   5.33 &   4.86 &  0.24 &  0.81 &   1.40 &  5430.90 &  0.25 &  2.75 &   6.73 &  3394.56 &  0.24 &  2.75 &   5.05 &   5.32 &  0.28 &  0.81 &   1.22 &  6896.05 &  0.25 &  2.75 &   6.15 &  2546.43 &  0.21 &  2.75 &   4.91 &  5.66 &  0.34 &  0.81 &   1.13 \\
		    &     & 6 &  2886.13 &  0.19 &  1.67 &   4.36 &   146.48 &  0.17 &  1.93 &   2.77 &   0.57 &  0.17 &  0.19 &   0.79 &  1651.91 &  0.21 &  1.67 &   4.20 &    15.06 &  0.22 &  1.93 &   2.29 &   0.71 &  0.21 &  0.19 &   0.81 &    93.66 &  0.19 &  1.46 &   3.64 &    19.36 &  0.19 &  1.93 &   2.02 &  0.55 &  0.18 &  0.12 &   0.40 \\
		    &     & 9 &  1235.12 &  0.32 &  2.05 &   2.59 &   342.32 &  0.22 &  0.96 &   1.22 &   1.99 &  0.21 &  0.22 &   0.26 &   404.70 &  0.29 &  1.90 &   2.12 &    28.09 &  0.21 &  0.88 &   0.76 &   0.82 &  0.20 &  0.13 &   0.17 &    99.73 &  0.21 &  1.99 &   1.58 &     2.23 &  0.22 &  0.39 &   0.58 &  0.87 &  0.20 &  0.06 &   0.08 \\
		    & 100 & 3 &   703.05 &  0.23 &  2.11 &   4.15 &   373.65 &  0.22 &  2.03 &   2.70 &   1.42 &  0.22 &  0.47 &   0.63 &   731.09 &  0.22 &  2.11 &   2.92 &   157.29 &  0.20 &  2.03 &   1.96 &   1.11 &  0.22 &  0.54 &   0.47 &   596.88 &  0.27 &  2.06 &   2.29 &   349.78 &  0.30 &  2.13 &   1.58 &  3.22 &  0.29 &  0.53 &   0.44 \\
		    &     & 6 &  7222.95 &  0.22 &  0.60 &   3.43 &  1814.25 &  0.18 &  0.48 &   2.08 &  22.11 &  0.21 &  0.13 &   0.44 &  7217.64 &  0.14 &  0.47 &   2.57 &   916.42 &  0.21 &  0.48 &   1.75 &   7.04 &  0.24 &  0.28 &   0.27 &  7216.94 &  0.15 &  0.47 &   2.11 &   656.48 &  0.20 &  0.37 &   1.41 &  7.89 &  0.22 &  0.24 &   0.20 \\
		    &     & 9 &  3321.23 &  0.34 &  0.07 &   0.28 &    16.40 &  0.20 &  0.02 &   0.18 &   2.33 &  0.17 &  0.08 &   0.08 &   198.14 &  0.19 &  0.07 &   0.32 &    14.84 &  0.21 &  0.02 &   0.13 &   2.31 &  0.21 &  0.08 &   0.05 &    47.16 &  0.23 &  0.08 &   0.33 &     9.77 &  0.21 &  0.01 &   0.12 &  2.63 &  0.20 &  0.06 &   0.04 \\
		\hline
	\end{tabular}}
\end{table}
\end{landscape}
\endgroup
\restoregeometry
\newgeometry{bottom=3cm,top=3cm}
\begin{table}[ht!]
	\centering
	\caption{All instances of second experiment. $|I|=100,|J|=25,|K|=6,r=0.5,\beta=0.1$}
	\label{tab:set6:alldata}
	\begin{tabular}{rrrr||rrrr}\toprule
		 $t_\text{MSP}$ &  $t_\text{MIP}$ & $\Delta_\text{avg}$ & $\Delta_\text{tail}$ & $t_\text{MSP}$ &  $t_\text{MIP}$ & $\Delta_\text{avg}$ & $\Delta_\text{tail}$ \\
		\midrule
		  31.15 & 0.23 & 1.53 & 3.20 &    2.15 & 0.16 & 2.24 & 2.09 \\
		   1.92 & 0.21 & 1.66 & 6.17 &   20.09 & 0.16 & 1.80 & 6.14 \\
		   8.75 & 0.24 & 0.52 & 3.07 &    7.18 & 0.16 & 2.13 & 1.93 \\
		  28.06 & 0.23 & 5.08 & 2.86 &    1.02 & 0.16 & 3.03 & 3.61 \\
		   1.36 & 0.30 & 1.00 & 1.80 &    3.58 & 0.24 & 1.81 & 6.12 \\
		   3.67 & 0.20 & 2.27 & 2.50 &    3.64 & 0.19 & 1.19 & 3.07 \\
		   2.00 & 0.20 & 2.51 & 2.03 &  128.69 & 0.23 & 3.27 & 2.98 \\
		 192.11 & 0.16 & 2.61 & 8.23 &    0.89 & 0.18 & 1.45 & 0.93 \\
		   0.94 & 0.20 & 0.43 & 2.23 &    1.62 & 0.23 & 1.85 & 3.56 \\
		   0.80 & 0.18 & 1.64 & 2.55 &    4.19 & 0.22 & 2.10 & 1.97 \\
		  16.40 & 0.19 & 2.23 & 2.45 &    2.16 & 0.19 & 0.16 & 1.46 \\
		   1.21 & 0.18 & 2.82 & 1.50 &    1.46 & 0.24 & 2.48 & 2.00 \\
		   1.79 & 0.20 & 0.72 & 2.77 &    0.69 & 0.20 & 1.79 & 2.54 \\
		  21.78 & 0.21 & 4.50 & 4.61 &   20.73 & 0.20 & 2.26 & 3.50 \\
		   1.35 & 0.19 & 0.69 & 0.86 &    1.86 & 0.24 & 1.77 & 2.63 \\
		  31.11 & 0.19 & 0.98 & 3.21 &   14.92 & 0.17 & 1.99 & 8.57 \\
		   8.44 & 0.19 & 1.82 & 3.81 &    0.78 & 0.20 & 0.85 & 1.92 \\
		   1.75 & 0.21 & 0.88 & 0.92 &   10.48 & 0.23 & 2.50 & 2.29 \\
		   1.94 & 0.21 & 2.18 & 2.65 &    1.63 & 0.24 & 2.08 & 2.29 \\
		   0.98 & 0.20 & 0.87 & 3.27 &   10.78 & 0.18 & 0.34 & 1.80 \\
		  27.72 & 0.22 & 2.03 & 5.20 &   38.80 & 0.20 & 1.96 & 4.69 \\
		  14.72 & 0.15 & 3.34 & 0.99 &   19.74 & 0.24 & 0.65 & 2.20 \\
		   0.67 & 0.24 & 0.81 & 2.69 &    1.37 & 0.30 & 2.82 & 2.92 \\
		   3.54 & 0.20 & 2.64 & 2.75 &    6.28 & 0.19 & 2.02 & 2.08 \\
		   6.37 & 0.21 & 2.79 & 6.35 &   22.27 & 0.34 & 1.91 & 3.13 \\
		   1.86 & 0.23 & 0.93 & 2.09 &    1.69 & 0.20 & 2.21 & 2.42 \\
		   1.54 & 0.20 & 2.00 & 3.45 &   27.77 & 0.19 & 0.76 & 3.28 \\
		  40.16 & 0.17 & 2.06 & 3.44 &    2.00 & 0.21 & 2.57 & 1.93 \\
		   7.23 & 0.21 & 3.17 & 3.17 &    2.61 & 0.18 & 2.14 & 3.33 \\
		   5.77 & 0.17 & 2.84 & 1.98 &   40.93 & 0.18 & 1.53 & 4.61 \\
		   2.10 & 0.19 & 1.39 & 3.00 &   18.84 & 0.16 & 0.89 & 4.37 \\
		 404.70 & 0.19 & 1.50 & 2.82 &   11.26 & 0.16 & 3.98 & 4.76 \\
		  24.26 & 0.18 & 4.81 & 3.42 &   14.41 & 0.18 & 1.82 & 5.87 \\
		   0.76 & 0.20 & 1.28 & 3.88 &   12.14 & 0.16 & 2.75 & 2.75 \\
		   0.64 & 0.20 & 0.87 & 1.39 &   12.58 & 0.17 & 1.42 & 3.46 \\
		   0.97 & 0.23 & 1.77 & 2.19 &    0.84 & 0.18 & 0.41 & 2.15 \\
		   0.53 & 0.18 & 1.95 & 2.04 &    5.20 & 0.19 & 3.80 & 2.02 \\
		   7.24 & 0.22 & 2.21 & 1.68 &   28.16 & 0.15 & 4.68 & 3.56 \\
		   0.87 & 0.25 & 0.71 & 1.42 &   39.10 & 0.16 & 3.47 & 3.59 \\
		   8.51 & 0.20 & 2.48 & 4.06 &   19.22 & 0.16 & 3.78 & 3.13 \\
		  13.06 & 0.20 & 4.44 & 2.80 &    0.56 & 0.20 & 0.63 & 3.22 \\
		  59.78 & 0.20 & 5.67 & 4.91 &    0.68 & 0.17 & 1.92 & 2.44 \\
		  67.50 & 0.19 & 2.96 & 3.02 &    0.70 & 0.17 & 1.86 & 1.40 \\
		   3.80 & 0.17 & 0.79 & 1.20 &   15.20 & 0.17 & 0.78 & 2.66 \\
		   3.25 & 0.20 & 2.24 & 1.57 &    0.88 & 0.17 & 2.31 & 2.13 \\
		   5.23 & 0.16 & 1.14 & 4.91 &    0.59 & 0.15 & 1.78 & 1.68 \\
		   0.71 & 0.17 & 0.89 & 3.01 &    1.08 & 0.20 & 2.21 & 3.14 \\
		   4.19 & 0.17 & 3.09 & 2.32 &    1.14 & 0.17 & 0.76 & 2.48 \\
		   3.53 & 0.18 & 1.37 & 6.33 &    1.58 & 0.19 & 1.45 & 3.14 \\
  19.99 & 0.14 & 3.48 & 5.05 &   13.48 & 0.18 & 1.71 & 5.41 \\
		\bottomrule
	\end{tabular}
\end{table}

\restoregeometry
\end{document}